\documentclass[11pt,a4paper]{article}%
\usepackage[T1]{fontenc}
\usepackage[utf8]{inputenc}
\usepackage{authblk}
\usepackage{amssymb, amsmath}
\usepackage{amsthm, amsfonts,mathrsfs}
\usepackage{mathtools}
\usepackage{times}
\usepackage{graphicx}
\usepackage{subfigure}
\usepackage{cite}
\usepackage{hyperref}
\usepackage{epsfig}
\usepackage{xcolor}
\usepackage[all]{xypic}
\usepackage{comment}
\usepackage{tikz}
\usepackage{pgfplots}
\usepackage{float}
\pgfplotsset{compat=1.18}
\usepgfplotslibrary{fillbetween} 
\hypersetup{colorlinks=true,citecolor=teal,linkcolor=teal}
\newtheorem{thm}{Theorem}[section]

\newtheorem{lemma}[thm]{Lemma}

\newtheorem{remark}[thm]{Remark}

\numberwithin{equation}{section}
\newcommand{\ep}{\varepsilon}
\newcommand{\al}{\alpha}
\newcommand{\ka}{\kappa}

\newcommand{\xe}{X^\varepsilon}
\newcommand{\ye}{Y^\varepsilon}

\newcommand{\E}{\mathbb{E}}

\newcommand{\F}{\mathcal{F}}
\newcommand{\R}{\mathbb{R}}
\newcommand{\N}{\mathbb{N}}
\newcommand{\intot}{\int_0^t}
\newcommand{\intr}{\int_{\mathbb{R}-\{0\}}}

\newcommand{\barf}{\bar{f}}
\newcommand{\la}{\langle}
\newcommand{\ra}{\rangle}

\newcommand{\sgn}{\operatorname{sgn}}
\newcommand{\law}{\operatorname{Law}}
\newcommand{\bp}{\bar{\Phi}}
\newcommand{\hp}{\widehat{\Phi}}
\newcommand{\p}{\Phi}
\newcommand{\pr}{\prime}
\newcommand{\tm}{\widetilde{M}}

\newcommand{\T}{\Theta}
\newcommand{\M}{\mathcal{M}}
\newcommand{\bm}{\mathbf{M}}
\newcommand{\tT}{\widetilde{\Theta}}
\newcommand{\vp}{\varphi}

\newcommand{\LL}{\mathcal{L}}

\title{\textbf{Stable and Gaussian Fluctuation Limit of a L\'evy-Driven Slow-Fast System}\thanks{This work is supported by NSFC Grant No.  12371243.}}

\author[1]{Qingming Zhao\thanks{zhao.qingming@foxmail.com}}

\author[2]{Xueru Liu\thanks{327625941@qq.com}}

\author[1]{Wei Wang\thanks{Corresponding author: wangweinju@nju.edu.cn}}

\affil[1]{School of Mathematics, Nanjing University, Nanjing 210093, P. R. China}
\affil[2]{School of Mathematical Sciences, Dalian University of Technology, 116024 Dalian, P. R. China}

\date{} 

\begin{document}
	\maketitle

	\noindent{\small{\hspace{1.1cm} }}

	\noindent \textbf{Abstract~~~} We study fluctuation limit of a slow-fast system driven by $\alpha$-stable L\'evy noise with $1<\alpha<2.$ The slow component is  generated by an odd polynomial function $f(y):=y^q,$ while in the fast component, the drift is $g(y):=-|y|^p\operatorname{sgn}(y)$ for some $p>0.$ Although the noise is given, the fluctuation limit is either a stable process or a Brownian motion, depending on both $p$ and $\al.$ The critical line between stable limit and Brownian limit is $q+1-p=\alpha/2.$   
	\\[2mm]
	\textbf{Keywords~~~}{Slow-Fast System}, {Central Limit Theorem}, {Averaging Principle}, {Stable Process}
	\\[2mm]
	\\
	\textbf{2020 Mathematics Subject Classification~~~}60F05, 60G51, 60H10

\section{Introduction}
For the L\'evy-driven slow-fast system of the form
\begin{equation*}
	\begin{cases}
		\dot{x}^\ep_t=f_1(x^\ep_t,y^\ep_t)+c_1\dot{L}^1_t,\enspace x^\ep_0=x_0,\\
		\dot{y}^\ep_t=\frac{1}{\ep}f_2(x^\ep_t,y^\ep_t)+c_2 \dot{L}^{2,\ep}_t,\enspace y^\ep_0=y_0,
	\end{cases}
\end{equation*}
where $L^{2,\ep}$ is a suitable re-scaled process of the L\'evy process $L^2,$ there are classical questions: (1) whether $x^\ep\to \bar x$ for some random process $\bar x$; (2) is there a $\gamma>0$ such that $\frac{x^\ep-\bar x}{\ep^\gamma}\to z$ for some non-zero random process $z.$ The first question is known as averaging, which is linked to law of large numbers, while the second one is known as fluctuation, which is linked to (generalized) central limit theorem. 

In order to obtain $\bar x$, if $L^1$ and $L^2$ are $\al$-stable processes with $\al\in(1,2),$ $L^{2,\ep}$ should be taken as $\ep^{-1/\al}L^2$~\cite[e.g.]{StableAveraging}. Besides, to obtain a non-zero $z,$ if $L^1$ and $L^2$ are L\'evy processes with truncated jumps, setting $L^{2,\ep}_t=L^2_{t/\ep},$ Yang et al.~\cite{Fluc-Levy}  proved $\gamma$ should be $1/2,$ and the limit is a SDE driven by Brownian motion. This is quite similar to the Brownian-driven case, which also requires $\gamma=1/2$~\cite[Section 5.4]{DW14}.

A natural question is that, whether we can obtain a limit which is a jump process. As we know, stable law can arise in the limit of a sequence of random variables with infinite variance~\cite{BRE92}. So it is reasonable to take the driving noise as $\al$-stable process and to expect the limit is again a stable process. Also note that there has been some work obtaining stable limit from discrete dynamical systems~\cite[e.g.]{CFKMI20,IR15}.

Furthermore, in the aforementioned work \cite{StableAveraging,Fluc-Levy}, as well as most of the references therein, some restrictive assumptions are imposed, such as Lipschitz continuity or boundedness of the drift. In our article, we do not impose such abstract assumptions, and instead, we take $f$ and $g$ to be very concrete functions, which in general do not satisfy those classical assumptions.

Another subtle point is that, in most of the existing literature about averaging, the fast component of the slow-fast system is assumed to be exponentially mixing. However, we will encounter the situation where the fast system seems only polynomially mixing, so that parts of classical approaches in averaging are invalid.

To be precise, we set
\[f(y):=y^q,\quad g(y):=-|y|^p\sgn(y),\]
 where $q$ is a positive odd integer,  $p>0,$ and we assume that $q<\al+p-1.$ Let $L$ be a symmetric $\al$-stable L\'evy process with $1<\al<2.$ It is known that the L\'evy measure $\nu$ of $L$ is \[\nu(dx)=c_\al\frac{1}{|x|^{1+\al}}dx, \] where $c_\al>0$ is a constant, and we assume $c_\al=1$ for notation simplicity. The Poisson random measure and compensated Poisson random measure of $L$ are denoted by $N$ and $\tilde N,$ respectively. Then the equation
\begin{equation}\label{frozen}
	\dot{Y}=g(Y)+\dot{L}
\end{equation}
has a unique strong solution and admits a unique invariant measure $\mu$ (see Lemma~\ref{lemma-wellposed}--\ref{lemma-kulic}).

Throughout the paper, for each probability measure $\lambda$ on $\R,$ we write $Y^\lambda$ to mean the solution of \eqref{frozen} with initial distribution $\lambda,$ and we write $Y^y:=Y^{\delta_y}$ for each deterministic $y\in\R.$ With a slight abuse of notation, we simply write $Y$ to mean $Y^\mu.$ Besides, we assume that $Y_0$  is independent of $(L_t)_{t\geq 0}.$ Set the filtration
\[\mathcal{F}_t:=\sigma(Y_0, L_s:0\leq s\leq t)\enspace\text{with the usual augmentation},\quad t\geq 0.\]
The independence of $Y_0$ and $(L_t)_{t\geq 0}$ makes $L$ a $(\mathcal{F}_t)_t$ $\al$-stable L\'evy process. Also note that $Y$ is a strong Markov process, and its semigroup is denoted as $(P_t)_{t\geq 0}.$
Now we consider
\begin{equation}\label{slow-fast}
	\begin{cases}
		\dot{X}^\ep=f(\ye),\enspace\xe_0=\xi,\\
		\dot{Y}^\ep=\frac{1}{\ep}g(\ye)+\frac{1}{\ep^{1/\al}}\dot{L},\enspace \ye_0=\eta.
	\end{cases}
\end{equation}
Here $\xi,\eta\in\R$ are almost surely finite random variables independent of $L$. The averaged equation is
\begin{equation*}
	X_t=\xi+t\barf.
\end{equation*}
 where $\barf:=\int_\R f(y)\mu(dy).$ This integral is absolutely convergent by Lemma \ref{lemma-kulic} (i) and the assumption $q+1-p<\al.$ But since $f$ is odd and $\mu$ is symmetric, we have $\barf=0,$ so 
 \begin{equation*}
 	X_t=\xi.
 \end{equation*}
 Our main result is the following theorem.
 \begin{thm}\label{main}
 	Let $p>0$ and $q$ be a positive odd integer such that $p\neq q+1.$ Suppose that
 	\[P(|\xi|+|\eta|<\infty)=1,\] and that $\eta$ is independent of $L.$ 
 	
 	(i) If $q+1-\al<p<q+1-\frac{\al}{2},$  there exists an $\frac{\al}{q+1-p}$-stable process $S$ such that 
 	\[\ep^{-(1-\frac{q+1-p}{\al})}(\xe-X)\to S\quad\text{in f.d.d. as $\ep\to0$.}\]  
 	
 	(ii) If $p> q+1-\frac{\al}{2},$ there exists a  Brownian motion $W$ with variance $\sigma^2,$ defined in \eqref{def-of-sigma2}, such that \[\ep^{-1/2}(\xe-X)\to W\quad\text{in f.d.d. as $\ep\to0.$}\]  
 	
 	(iii) If $p=q+1-\frac{\al}{2},$ there exists a standard Brownian motion $B$ such that \[\frac{\xe-X}{\sqrt{\ep\log(1/\ep)}}\to (q+1-p)^{-3/2}B\quad\text{in f.d.d. as $\ep\to0.$}\]
 \end{thm}

 	\begin{figure}[H]\label{fig}
 		\centering
 		
 		\foreach \qq/\pmax in {1/3,3/5.5} {%
 			\begin{minipage}{0.48\textwidth}
 				\centering
 				\begin{tikzpicture}
 					\pgfmathsetmacro{\pexcluded}{\qq+1}
 					\pgfmathsetmacro{\ystable}{\qq+1-1.5+0.3}
 					\pgfmathsetmacro{\ygaussian}{\qq+1+0.5}
 					
 					\begin{axis}[
 						width=\linewidth,
 						height=6.5cm,
 						axis lines=left,
 						xlabel={$\alpha$},
 						ylabel={$p$},
 						title={$q=\qq$},
 						xmin=0.95, xmax=2.05,
 						ymin=0, ymax=\pmax,
 						xtick={1,1.5,2},
 						samples=160,
 						clip=true
 						]
 						
 						\addplot[
 						name path=critical,
 						draw=none,
 						domain=1:2
 						] {\qq+1-x/2};
 						
 						\addplot[
 						name path=top,
 						draw=none,
 						domain=1:2
 						] {\pmax};
 						
 						\addplot[
 						draw=none,
 						fill=orange,
 						fill opacity=0.3
 						] fill between[of=critical and top];
 						
 						\addplot[
 						name path=lower,
 						draw=none,
 						domain=1:2
 						] {\qq+1-x};
 						
 						\addplot[
 						draw=none,
 						fill=cyan,
 						fill opacity=0.3
 						] fill between[of=lower and critical];
 						
 						\addplot[
 						blue!70!black,
 						thick,
 						dashed,
 						domain=1:2
 						] {\qq+1-x};
 						
 						\addplot[
 						red!75!black,
 						very thick,
 						domain=1:2
 						] {\qq+1-x/2};
 						
 						\addplot[
 						black,
 						thick,
 						dashed
 						] coordinates {
 							(1,\pexcluded) (2,\pexcluded)
 						};
 						
 						\node[
 						font=\scriptsize,
 						fill=white,
 						inner sep=2pt,
 						anchor=south
 						] at (axis cs:1.5,\pexcluded)
 						{Excluded: $p=\pexcluded$};
 						
 						\draw[gray,densely dotted]
 						(axis cs:1,0) -- (axis cs:1,\pmax);
 						\draw[gray,densely dotted]
 						(axis cs:2,0) -- (axis cs:2,\pmax);
 						
 						\node[font=\small]
 						at (axis cs:1.5,\ystable) {Stable};
 						
 						\node[font=\small]
 						at (axis cs:1.5,\ygaussian) {Gaussian};
 						
 					\end{axis}
 				\end{tikzpicture}
 			\end{minipage}
 			\hfill
 		}
 		
 		\caption{Parameter regions in Theorem~\ref{main}.
 			The cyan regions correspond to the stable limit:
 			$2-\alpha<p<2-\alpha/2$ when $q=1$, and
 			$4-\alpha<p<4-\alpha/2$ when $q=3$.
 			The orange regions correspond to the Gaussian limit
 			$p>q+1-\alpha/2$, with $p=q+1$ excluded.
 			The solid red curves correspond to the critical Gaussian
 			limit with logarithmic scaling.}
 		\label{fig:parameter-regions}
 	\end{figure}

From Figure \ref{fig} one sees that under the assumptions in Theorem \ref{main} and $q\in 2\N-1$,
\begin{equation}\label{small-p}
	p\in (0,1)\implies q=1.
\end{equation}
 Indeed, if by contradiction $q\geq 3$, then $p>q+1-\al\geq 4-\al>2$, a contradiction to $p<1$.
 
 By \eqref{slow-fast}, we have
 \begin{equation}
 	\xe_t-X_t=\intot f(\ye_s)ds.
 \end{equation}
By scaling property of $L$, $(\ye_t)_{t\geq 0}\overset{d}{=}(Y^\eta_{t/\ep})_{t\geq 0},$ and by a change-of-variable,
\[ \intot f(Y^\eta_{s/\ep})ds=\ep\int_0^{t/\ep}f(Y^\eta_r)dr,  \]
so we have
\begin{equation}\label{distribution0}
	 (\xe_t-X_t)_{t\geq 0}\overset{d}{=}\Big(\intot f(Y^\eta_{s/\ep})ds\Big)_{t\geq 0}=\Big(\ep\int_0^{t/\ep}(Y^\eta_r)^q dr\Big)_{t\geq 0}.
\end{equation}
Multiplying by $\ep^{-\gamma},$
\begin{equation}\label{distribution}
  (\ep^{-\gamma}(\xe_t-X_t))_{t\geq 0}\overset{d}{=}\Big(\ep^{1-\gamma}\int_0^{t/\ep}(Y^\eta_r)^qdr\Big)_{t\geq 0}.
\end{equation}
Let $T:=1/\ep$ and $A^\eta_t:=\intot (Y^\eta_s)^qds.$ To prove our main results, it suffices to show:

(i) If $q+1-\al<p<q+1-\frac{\al}{2}$,   then 
\begin{equation}\label{real-equiv-thm-i}
	\big(T^{-(q+1-p)/\al}A^\eta_{Tt}\big)_{t\geq 0}\to (S_t)_{t\geq 0} \quad\text{in f.d.d,}
\end{equation}  
where $S$ is an  $\frac{\al}{q+1-p}$-stable process. This means $$\Big(\ep^{(q+1-p)/\al}\int_0^{t/\ep}(Y^\eta_r)^qdr\Big)_{t\geq 0}\to (S_t)_{t\geq 0}\quad\text{in f.d.d,}$$  and by \eqref{distribution}, $1-\gamma=\frac{q+1-p}{\al},$ so $\gamma=1-\frac{q+1-p}{\al}.$

(ii) If $p>q+1-\frac{\al}{2}$, then 
\begin{equation}\label{real-equiv-thm-ii}
	(T^{-1/2}A^\eta_{Tt})_{t\geq 0}\to (W_t)_{t\geq 0} \quad\text{in f.d.d,}
\end{equation} 
where $W$ is a Brownian motion with variance $\sigma^2$ defined in \eqref{def-of-sigma2}. This means $$(\ep^{1/2}\int_0^{t/\ep}(Y^\eta_r)^qdr)_{t\geq 0}\to (W_t)_{t\geq 0}\quad\text{in f.d.d,}$$ and by \eqref{distribution}, $1-\gamma=\frac{1}{2}$, so $\gamma=1/2$.

(iii) If $p=q+1-\frac{\al}{2},$
\begin{equation}\label{real-equiv-thm-iii}
	\Big(\frac{A^\eta_{Tt}}{\sqrt{T\log T}}\Big)_{t\geq 0}\to ((q+1-p)^{-3/2}B_t)_{t\geq 0}\quad\text{in f.d.d,}
\end{equation} 
where $B$ is a standard Brownian motion, so by \eqref{distribution0},
\[ \Big(\frac{\xe_t-X_t}{\sqrt{\ep\log(1/\ep)}}\Big)_{t\geq 0}\overset{d}{=}\Big(\frac{A^\eta_{t/\ep}}{\sqrt{(1/\ep)\log(1/\ep)}}\Big)_{t\geq 0}\to ((q+1-p)^{-3/2}B_t)_{t\geq 0}. \]

However, we will not directly prove \eqref{real-equiv-thm-i}--\eqref{real-equiv-thm-iii}. Instead, we first consider their stationary version. That is, setting $A_t:=\intot (Y_s)^qds,$ with $Y$  the solution to  \eqref{frozen} with initial value distributed as~$\mu$\,, we prove that  
\begin{equation}\label{equiv-thm-i}
	\big(T^{-(q+1-p)/\al}A_{Tt}\big)_{t\geq 0}\to (S_t)_{t\geq 0} \quad\text{in f.d.d.}
\end{equation}
when $q+1-\al<p<q+1-\frac{\al}{2},$
\begin{equation}\label{equiv-thm-ii}
	(T^{-1/2}A_{Tt})_{t\geq 0}\to (W_t)_{t\geq 0} \quad\text{in f.d.d.}
\end{equation}
when $p>q+1-\frac{\al}{2},$ and
\begin{equation}\label{equiv-thm-iii}
	\Big(\frac{A_{Tt}}{\sqrt{T\log T}}\Big)_{t\geq 0}\to ((q+1-p)^{-3/2}B_t)_{t\geq 0}\quad\text{in f.d.d.}
\end{equation}
when $p=q+1-\frac{\al}{2}.$ Then we remove the stationary assumption which is precisely our main results.  
 
We next introduce some auxiliary functions and derive a decomposition of $A_t$ that will be used in the proofs in the subsequent sections. Throughout the paper, we denote $r:=q+1-p.$ Since we assumed $p\neq q+1,$ we have $r\neq 0.$  Define $\p(0):=0$ and
 \begin{equation}\label{def-of-phi}
 	\p(y)=-\frac{1}{r}|y|^r\sgn(y),\quad y\in\R-\{0\}.
 \end{equation}
 Choose odd $\hp\in C^3(\R)$ such that 
 \[\hp(y)=\p(y),\quad \text{for all}\enspace y\in\R\enspace\text{with}\enspace |y|\geq 1. \]
 Define
 \[\bp(y):=\hp(y)-\p(y),\enspace \bar{g}(y):=g(y)\hp^\pr(y)-f(y),\quad y\in\R, \] and $\bar g(0):=0.$
   Applying the It\^o's  formula to the stationary solution $Y$ of \eqref{frozen},
\begin{equation}\label{ito}
	\begin{aligned}[b]
		&\quad\hp(Y_t)-\hp(Y_0)\\&=\intot\hp^\pr(Y_s) g(Y_s)ds+\intot\intr \hp(Y_{s-}+x)-\hp(Y_{s-})\tilde{N}(dsdx)\\&\quad+\intot\intr\hp(Y_{s-}+x)-\hp(Y_{s-})-x\hp^\pr(Y_s)\nu(dx)ds.
	\end{aligned}
\end{equation}
  By symmetric property of $\nu$,
\begin{equation*}
	\intr \hp(y+x)-\hp(y)-x\hp^\pr(y)\nu(dx)=\int_0^\infty \hp(y+x)+\hp(y-x)-2\hp(y)\nu(dx),
\end{equation*}
substituting which into \eqref{ito} yields
\begin{equation}\label{ito2}
	\begin{aligned}[b]
		&\quad\hp(Y_t)-\hp(Y_0)\\&=\intot\hp^\pr(Y_s) g(Y_s)ds+\intot\intr \hp(Y_{s-}+x)-\hp(Y_{s-})\tilde{N}(dsdx)\\&\quad+\intot\int_0^\infty\hp(Y_{s-}+x)+\hp(Y_{s-}-x)-2\hp(Y_{s-})\nu(dx)ds.
	\end{aligned}
\end{equation}
The definition of $\bar{g}$ gives $f=g\hp^\pr-\bar{g},$ so by \eqref{ito2},
\begin{equation}\label{new-int-f}
	\begin{aligned}[b]
		&\quad\intot f(Y_s)ds\\&=\intot g(Y_s)\hp^\pr(Y_s)ds-\intot\bar{g}(Y_s)ds\\&=\hp(Y_t)-\hp(Y_0)-\intot\intr \hp(Y_{s-}+x)-\hp(Y_{s-})\tilde{N}(dsdx)\\&\quad-\intot\int_0^\infty\hp(Y_{s-}+x)+\hp(Y_{s-}-x)-2\hp(Y_{s-})\nu(dx)ds-\intot\bar{g}(Y_s)ds\\&=:\hp(Y_t)-\hp(Y_0)-M_t-\intot R(Y_s)ds,
	\end{aligned}
\end{equation}
where
\begin{equation*}
	M_t:=\intot\intr \hp(Y_{s-}+x)-\hp(Y_{s-})\tilde{N}(dsdx),
\end{equation*}
and
\begin{equation*}
	R(y):=\bar{g}(y)+\int_0^\infty \hp(y+x)+\hp(y-x)-2\hp(y)\nu(dx).
\end{equation*}

\begin{remark}\label{remark-odd-bound}

	(i) Functions $f,g,\bar{g},\p,\hp,\bp$ and $R$ are all odd. The odd property of $R$ comes from its definition and the symmetric property of $\nu$.

	(ii) The definition of $\hp$ implies that $\bp$ is compactly supported. When $r>0,$ $\bp$ is continuous due to the continuity of $\p,$ so $\bp$ is bounded. Besides,  a direct computation gives
	\[g(y)\p^\pr(y)=f(y),\quad y\in\R-\{0\},\]
	so \[\bar{g}(y)=g(y)\hp^\pr(y)-g(y)\p^\pr(y)=g(y)\bp^\pr(y),\quad y\in \R-\{0\}, \] which implies that $\bar{g}$ vanishes outside $[-1,1].$  Moreover, the definition of $\bar g$ implies that it is continuous, so $\bar g$ is bounded whenever $r\neq 0.$  
	
	(iii) It is expected that Theorem \ref{main} (ii) is valid in $r=0$ case, although we exclude this case in our paper. Defining
	$\Phi_0(y):=-\sgn(y)\log|y|,$ and replacing $\Phi$ by $\Phi_0,$ a similar proof with that in Section \ref{subsection-r<alpha/2} is expected to work. However, several quantitative estimates have to be re-written, such as Lemma \ref{lemma-Holder} and Lemma \ref{lemma-technique}. This is left to subsequent work.
	
	(iv) In Theorem \ref{main}, the restriction $p>q+1-\al$ is necessary. Indeed, by Lemma \ref{lemma-kulic} (i), if $p\leq q+1-\al,$ then $\bar f$ is not well-defined as a Lebesgue integral, although its principle value is $0$ due to symmetry.
\end{remark} 
Our paper is organized as follows.  In Section \ref{section-stable-limit}, we demonstrate \eqref{equiv-thm-i}. We extract a stable process from the pre-limit, and show the rest of pre-limit vanishes as $T$ tends to infinity. In Section~\ref{section-Brownian-limit}, we prove \eqref{equiv-thm-ii} and \eqref{equiv-thm-iii}. To this aim, a central limit theorem for c\`adl\`ag local martingale is applied. In the critical case $r=\al/2,$ a more subtle estimate is given to confirm the logarithmic scaling. In Section \ref{section-extension}, we drop the stationary assumption, that is, we prove \eqref{real-equiv-thm-i}--\eqref{real-equiv-thm-iii}. Some technical results are put in Appendix \ref{appendix}.

\section{Stable Limit}\label{section-stable-limit}

In this section, our principal assumption is
\begin{equation}\label{stable-assumption-1}
	q+1-\al<p<q+1-\frac{\al}{2}.
\end{equation}
Note that \eqref{stable-assumption-1} gives $\al/2<r<\al.$ Our goal is to show \eqref{equiv-thm-i}.
By \eqref{new-int-f},
\begin{equation}\label{star}
	\begin{aligned}[b]
		  T^{-r/\al}A_{Tt}&=T^{-r/\al}\int_0^{Tt}Y_s^qds\\&=T^{-r/\al}\hp(Y_{Tt})-T^{-r/\al}\hp(Y_0)\\&\quad-T^{-r/\al}M_{Tt}-T^{-r/\al}\int_0^{Tt}R(Y_s)ds\\&=:I_1(t,T)-I_2(t,T)-I_3(t,T)-I_4(t,T).
	\end{aligned}
\end{equation}
 So it suffices to analyze behaviors of terms in \eqref{star} as $T\to\infty$ separately.

\subsection{Vanishing limit of $I_1$ and $I_2$.}

Fix arbitrary $\delta>0$. Since $\law(Y_t)\equiv\mu$, $(\law(Y_{t}))_{t\geq 0}$ is tight in $\R$, that is, there is a constant $K>0$ such that
\begin{equation}\label{k-bound}
	P(|Y_{t}|>K)<\delta,\quad\text{for all}\enspace t\geq 0.
\end{equation}
 Since $\hp\in C^3(\R)$, it is bounded on compact set. Set $C_K:=\sup\limits_{-K\leq y\leq K}|\hp(y)|<\infty.$ Since $r>\al/2>0$,  $\lim\limits_{T\to\infty} T^{-r/\al}C_K=0.$ Therefore, there is $T_0>0$ such that
\[ T^{-r/\al}C_K<\delta,\quad\text{for all}\enspace T>T_0. \]
Therefore, by \eqref{k-bound}, for all $T>T_0$,
\begin{align*}
	P(T^{-r/\al}|\hp(Y_{Tt})|\geq \delta)\leq P(|Y_{Tt}|>K)\leq \delta.
\end{align*}
This means for each $t\geq 0,$
\begin{equation}\label{I1-result}
\lim\limits_{T\to\infty}	I_1(t,T)=\lim\limits_{T\to\infty}T^{-r/\al}|\hp(Y_{Tt})|= 0
\end{equation}
  in probability, and hence also in f.d.d. Besides, there is 
\begin{equation}\label{I2-result}
	\lim\limits_{T\to\infty}I_2(t,T)=\lim\limits_{T\to\infty}T^{-r/\al}|\hp(Y_0)|= 0
\end{equation}
a.s, and hence also in f.d.d.

\subsection{Stable limit of $I_3.$}\label{subsection-I3}

Let $\tm_t:=\intot\intr\p(x)\tilde{N}(dsdx),$ then
\begin{equation}
	\begin{aligned}[b]
		M_t&=\tm_t+\intot\intr \hp(Y_{s-}+x)-\hp(Y_{s-})-\p(x)\tilde{N}(dsdx)\\&=:\tm_t+\intot\intr h(Y_{s-},x)\tilde{N}(dsdx),
	\end{aligned}
\end{equation} 
where \[ h(y,x):=\hp(y+x)-\hp(y)-\p(x). \]
Therefore,
\begin{equation}\label{newI3}
	I_3(t,T)=T^{-r/\al}\tm_{Tt}+T^{-r/\al}\int_0^{Tt}\intr h(Y_{s-},x)\tilde{N}(dsdx)=:I_{31}(t,T)+I_{32}(t,T).
\end{equation}
 The characteristic function of $\tm$ is
\[\varphi_{\tm}(t,\lambda):=\E e^{i\lambda\tm_t}=e^{t\psi(\lambda)}, \]
where~\cite[(2.9)]{APP09}
\[ \psi(\lambda)=\intr [e^{i\lambda\p(x)}-1-i\lambda\p(x)]\nu(dx)\,.  \]
Therefore, the characteristic function of $T^{-r/\al}\tm_{Tt}$ with $T>0$ is
\[\E e^{i\lambda T^{-r/\al}\tm_{Tt}}=\varphi_{\tm}(Tt,\lambda T^{-r/\al})=e^{Tt\psi(\lambda T^{-r/\al})}. \]
For each $T>0,$ by change-of-variable $x=T^{1/\al}u$,  $\nu(dx)=\frac{1}{T}\nu(du)$, yielding that
\begin{align*}
	&\quad Tt\psi(\lambda T^{-r/\al})\\&=Tt\intr [e^{i\lambda T^{-r/\al}\p(x)}-1-i\lambda T^{-r/\al}\p(x)]\nu(dx)\\&=t\intr[ e^{i\lambda T^{-r/\al}\p(T^{1/\al}u)}-1-i\lambda T^{-r/\al}\p(T^{1/\al}u)]\nu(du)\\&=t\intr[ e^{i\lambda \p(u)}-1-i\lambda \p(u)]\nu(du)\\&=t\psi(\lambda),
\end{align*}
where we used $T^{-r/\al}\p(T^{1/\al}u)=\p(u)$ by the definition of $\p$ in the second-last equality. This precisely means
\[\E e^{i\lambda T^{-r/\al}\tm_{Tt}}=\E e^{i\lambda\tm_t},\quad\lambda\in\R,\]
so $T^{-r/\al}\tm_{Tt}\overset{d}{=}\tm_t$ for all $t\geq 0$. But both of the processes are L\'evy processes, so the self-similarity implies that $(T^{-r/\al}\tm_{Tt})_{t\geq 0}$ is an $\al/r$-stable process. That is to say,
\begin{equation}\label{I31-result}
	\law(-I_{31}(\cdot,T))=\law(S_\cdot),\quad T>0,
\end{equation}
where $S$ is an $\al/r$-stable process. Next we deal with $I_{32}.$ Write
\begin{equation}
	\begin{aligned}[b]
		&\quad I_{32}(t,T)\\&=T^{-r/\al}\int_0^{Tt}\int_{|x|<1} h(Y_{s-},x)\tilde{N}(dsdx)+T^{-r/\al}\int_0^{Tt}\int_{|x|\geq 1} h(Y_{s-},x)\tilde{N}(dsdx)\\&=:I_{32}^{<1}(t,T)+I_{32}^{\geq1}(t,T).
	\end{aligned}
\end{equation}
For each $t\geq 0$ and $1\leq\rho\leq 2$, by Burkholder--Davis--Gundy inequality \cite[Theorem 3.50]{PZ07}, 
\begin{align*}
	&\quad\E\Big|\intot\int_{|x|<1} h(Y_{s-},x)\tilde{N}(dsdx)\Big|^\rho\\&\leq C_\rho \E\Big[\int_0^\cdot\int_{|x|<1} h(Y_{s-},x)\tilde{N}(dsdx),\int_0^\cdot\int_{|x|<1} h(Y_{s-},x)\tilde{N}(dsdx)\Big]_t^{\rho/2}\\&= C_\rho\E\Big[\intot\int_{|x|<1}|h(Y_{s-},x)|^2N(dsdx)\Big]^{\rho/2}\\&\leq C_\rho\E\intot\int_{|x|<1}|h(Y_{s-},x)|^\rho N(dsdx)\\&=C_\rho\E\intot\int_{|x|<1}|h(Y_{s-},x)|^\rho \nu(dx)ds.
\end{align*}
Therefore,
\begin{equation}\label{rho-moment}
\begin{aligned}[b]
		\E|I_{32}^{<1}(t,T)|^\rho&=\E\Big|T^{-r/\al}\int_0^{Tt}\int_{|x|<1} h(Y_{s-},x)\tilde{N}(dsdx)\Big|^\rho\\&\leq C_\rho T^{-\rho r/\al}\E\int_0^{Tt}\int_{|x|<1}|h(Y_{s-},x)|^\rho \nu(dx)ds\\&= C_\rho T^{-\rho r/\al}\int_0^{Tt}\int_{|x|<1}\E|h(Y_{s-},x)|^\rho \nu(dx)ds\\&=C_\rho T^{-\rho r/\al}\int_0^{Tt}\int_{|x|<1}\int_\R|h(y,x)|^\rho\mu(dy)\nu(dx)ds\\&=C_\rho T^{1-\frac{\rho r}{\al}}t\int_{|x|<1}\int_\R|h(y,x)|^\rho\mu(dy)\nu(dx).
\end{aligned}
\end{equation}
If we show that there is a $\rho\in[1,2]$ such that
\begin{equation}\label{rho-bound}
	1-\frac{\rho r}{\al}<0,\quad\text{and}\enspace \int_{|x|<1}\int_{\R}|h(y,x)|^\rho\mu(dy)\nu(dx)<\infty,
\end{equation}
then from \eqref{rho-moment}, there is necessarily
	$\lim\limits_{T\to\infty}I_{32}^{<1}(t,T)=0$ in $ L^\rho$, and hence
\begin{equation}\label{I32-small}
	I_{32}^{<1}(\cdot,T)\to 0 \quad\text{in f.d.d.}
\end{equation}
  Similarly, if we find $\theta\in [1,2]$ such that
\begin{equation}\label{theta-bound}
	1-\frac{\theta r}{\al}<0,\quad\text{and}\enspace \int_{|x|\geq 1}\int_\R|h(y,x)|^\theta\mu(dy)\nu(dx)<\infty,
\end{equation}
then 
\begin{equation*}
	I_{32}^{\geq 1}(\cdot,T)\to 0 \quad \text{in f.d.d.}
\end{equation*}
Combining with \eqref{I32-small},
\[I_{32}(\cdot,T)=I_{32}^{<1}(\cdot,T)+I_{32}^{\geq 1}(\cdot,T)\to 0\quad\text{in f.d.d.}\]
Thus combining with \eqref{newI3} and \eqref{I31-result}, we have by Slutsky theorem that
\begin{equation}\label{I_3-result}
	\lim\limits_{T\to\infty}-I_3(\cdot,T)=S\quad\text{in f.d.d.}
\end{equation}

Now we state the proof strategy. We distinguish two cases: the first one is $q\geq 3$ or $p\geq 1$, and the second one is $q=1$ and $0<p<1.$ In the former case, we find $(\rho,\theta)\in [1,2]\times[1,2]$ such that \eqref{rho-bound} and \eqref{theta-bound} hold. In the latter case, we find $\theta\in [1,2]$ such that \eqref{theta-bound} holds, but the $L^\rho(\mu\otimes\nu)$ estimate \eqref{rho-bound} is in general invalid, so we instead prove the f.d.d. convergence \eqref{I32-small} by a localization argument.

\subsubsection{$q\geq 3$ or $p\geq 1$}

We always take $\rho=\theta=2.$

\noindent\textbf{Case 1: $|x|<1.$}

Taking square in the inequality in Lemma \ref{lemma-technique} (ii),
\[|h(y,x)|^2\leq C_r(|x|^2(2+|y|)^{2(r-1)_+}+|x|^{2r}),\quad y\in\R. \]
Note that 
\[\int_{|x|\leq 1}|x|^2\nu(dx)\leq C,\quad \int_{|x|\leq 1}|x|^{2r}\nu(dx)=2\int_0^1 x^{2r-\al-1}dx\leq C,  \]
where the last inequality comes from $2r-\al-1>\al-\al-1=-1,$
so
\begin{equation}\label{noname6}
	\int_{|x|\leq 1}|h(y,x)|^2\nu(dx)\leq C_r(2+|y|)^{2(r-1)_+}+C.
\end{equation}
Now we claim \begin{equation}\label{inequality}
	2(r-1)_+<\al+p-1.
\end{equation} 
Indeed, if $p\geq 1$ and $q=1,$ then $r-1=(q+1-p)-1=1-p\leq0,$ and \[2(r-1)_+=0<\al-1\leq \al+p-1.\] If $q\geq 3,$ when $r\leq 1$ \[\al+p-1>1+p-1=p>0=2(r-1)_+.\] So now we assume $q\geq 3$ and $r>1$. But $\al+p-1=\al+(q+1-r)-1=\al+q-r$, so it suffices to show $3r<\al+q+2$. Since $q\geq 3$, we have $2\al<4\leq q+2$, so $3r<3\al=\al+2\al<\al+q+2$, and hence \eqref{inequality} is shown. By \eqref{noname6}--\eqref{inequality} and Lemma \ref{lemma-kulic} (i),
\begin{equation}
	\int_\R\int_{|x|< 1}|h(y,x)|^2\nu(dx)\mu(dy)<\infty.
\end{equation}

\noindent\textbf{Case 2: $|x|\geq 1.$}

We first claim the inequality
\begin{equation}\label{inequality2}
	2r-\al<\al+p-1.
\end{equation}
Indeed, if $p\geq 1,$ we have
\[(\al+p-1)-(2r-\al)\geq \al-(2r-\al)=2(\al-r)>0. \]
If $q\geq 3$, since $\al+p-1=\al+(q+1-r)-1=\al+q-r$, it suffices to show $3r<2\al+q$. Since $q\geq 3$,  $\al< q$, and \[  r<\al=\frac{2\al}{3}+\frac{\al}{3}< \frac{2\al+q}{3},  \]
which is equivalent to $3r<2\al+q$.

\emph{(i)  $\al/2<r<1.$}

By Lemma \ref{lemma-technique} (iii), 
\[|h(y,x)|^2\leq C^2+C^2\min\{|x|^{2r},|y|^{2r}\}. \]
When $|y|<1,$
\begin{equation*}
	\int_{|x|\geq 1}|h(y,x)|^2\nu(dx)\leq C \int_{|x|\geq 1}1+|y|^{2r}\nu(dx)\leq C(1+|y|^{2r})\leq C(1+|y|^{2r-\al}).
\end{equation*}
When $|y|\geq 1,$
\begin{align*}
	&\quad \int_{|x|\geq 1}|h(y,x)|^2\nu(dx)\\&\leq C\int_{|x|\geq 1}1+\min\{|x|^{2r},|y|^{2r}\}\nu(dx)\\&=C+C\int_1^{|y|}|x|^{2r}\nu(dx)+C\int_{|y|}^\infty|y|^{2r}\nu(dx)\\&=C+C\int_1^{|y|}x^{2r-1-\al}dx+C|y|^{2r}\int_{|y|}^\infty x^{-1-\al}dx\\&\leq C+C|y|^{2r-\al}, 
\end{align*}
where we use $2r-1-\al>-1$ in the last inequality. So we have
\begin{equation*}
\int_{|x|\geq 1}|h(y,x)|^2\nu(dx)\leq C(1+|y|^{2r-\al}),\quad y\in\R.
\end{equation*}
 By \eqref{inequality2},
\begin{equation}
\int_\R\int_{|x|\geq 1}|h(y,x)|^2\nu(dx)\mu(dy)\leq\int_\R C(1+|y|^{2r-\al})\mu(dy)<\infty.
\end{equation}

\emph{(ii)  $1\leq r<\al.$}

\noindent When $|y|<1,$ by Lemma \ref{lemma-technique} (iv), and the fact $2r-3-\al<\al-3<-1,$
\begin{align*}
	\int_{|x|\geq 1}|h(y,x)|^2\nu(dx)&\leq C\int_{|x|\geq 1}1+|y|^2|x|^{2r-2}\nu(dx)\\&\leq C\Big(1+2|y|^2\int_1^\infty x^{2r-3-\al}dx\Big)\\&\leq C(1+|y|^2)\leq C(1+|y|^{2r-\al}),
\end{align*}
the last inequality coming from $(2r-\al)-2<2\al-\al-2=\al-2<0.$

When $|y|\geq 1,$ by Lemma \ref{lemma-technique} (iv),
\begin{align*}
	&\quad\int_{|x|\geq 1}|h(y,x)|^2\nu(dx)\\&\leq\int_{|x|\geq 1}C+C_r\min\{|x|^2|y|^{2r-2},|y|^2|x|^{2r-2}\}\nu(dx)\\&\leq C+C_r\int_{1\leq|x|<|y|}|x|^2|y|^{2r-2}\nu(dx)+C_r\int_{|x|\geq |y|}|y|^2|x|^{2r-2}\nu(dx)\\&\leq C+C|y|^{2r-2}\int_1^{|y|}x^{1-\al}dx+C|y|^2\int_{|y|}^\infty x^{2r-3-\al}dx\\&\leq C+C|y|^{2r-\al}+C|y|^{2r-\al}\leq C(1+|y|^{2r-\al}).
\end{align*}
By \eqref{inequality2} and Lemma \ref{lemma-kulic},
\begin{equation}
	\int_\R\int_{|x|\geq 1}|h(y,x)|^2\nu(dx)\mu(dy)\leq \int_\R C(1+|y|^{2r-\al})\mu(dy)<\infty.
\end{equation}

\subsubsection{$q=1$ and $0<p<1$.}
Since $q=1,$ $r=q+1-p=2-p\in (1,2).$

\noindent\textbf{Case 1: $|x|< 1.$}

By Lemma \ref{lemma-technique} (ii),
\[|h(y,x)|\leq C[|x|(2+|y|)^{r-1}+|x|^r],\quad0<|x|<1,\enspace y\in\R,\]
so
\begin{align}\label{h-integral-bound}
	\int_{|x|<1}|h(y,x)|^2\nu(dx)&\leq C\int_{|x|<1}[|x|^2(2+|y|)^{2r-2}+|x|^{2r}]\nu(dx)\nonumber\\&\leq C[1+(2+|y|)^{2r-2}].
\end{align}
For each $K,T>0,$ set $\tau_{T,K}:=\inf\{s\geq0:|Y_s|>KT^{1/\al}\},$ so that
\begin{equation}\label{local-bound}
	|Y_{s-}|\leq KT^{1/\al}, \quad s\in (0,\tau_{T,K}].
\end{equation}
By Lemma \ref{lemma-Y-L},
\[\{\sup_{0\leq s\leq Tt}|Y_s|\geq KT^{1/\al}\}\subset\{|Y_0|\geq\frac{K}{2}T^{1/\al}\}\cup\{\sup_{0\leq s\leq Tt}|L_s|\geq \frac{K}{4}T^{1/\al}\},\] so by self-similarity of $L,$
\begin{align*}
	P(\sup_{0\leq s\leq Tt}|Y_s|\geq KT^{1/\al})&\leq P(|Y_0|\geq\frac{K}{2}T^{1/\al})+P(\sup_{0\leq s\leq Tt}|L_s|\geq \frac{K}{4}T^{1/\al})\\&= P(|Y_0|\geq\frac{K}{2}T^{1/\al})+ P(\sup_{0\leq s\leq t}|L_s|\geq \frac{K}{4}).
\end{align*}
Since $L$ is c\`adl\`ag, it is locally bounded, and hence
\begin{equation}\label{noname5}
	\lim\limits_{K\to\infty}\limsup_{T\to\infty}P(\sup_{0\leq s\leq Tt}|Y_s|\geq KT^{1/\al})=0.
\end{equation}
Write
\[I_{32}^{<1}(t,T)=T^{-r/\al}\int_0^{Tt}\int_{|x|<1}h(Y_{s-},x)\tilde{N}(dsdx)=T^{-r/\al}J_{Tt},\] with \[J_t:=\int_0^t\int_{|x|<1}h(Y_{s-},x)\tilde{N}(dsdx).\]
$J$ is well-defined as a local martingale. Indeed, set $\sigma_n:=n\wedge \inf\{s>0:|Y_s|>n\},$ so that c\`adl\`ag property of $Y$ implies $\sigma_n\uparrow\infty$ almost surely. By \eqref{h-integral-bound}, 
\[\E\int_0^\infty\int_{|x|<1}|h(Y_{s-},x)|^2\mathbf{1}_{\{s\leq\sigma_n\}}\nu(dx)ds\leq Cn[1+(2+n)^{2r-2}]<\infty.\]
By It\^o isometry, \eqref{h-integral-bound} and \eqref{local-bound}, 
\begin{equation*}
	\begin{aligned}[b]
		\E|T^{-r/\al}J_{Tt\wedge\tau_{T,K}}|^2&= T^{-2r/\al}\E|\int_0^{Tt\wedge\tau_{T,K}}\int_{|x|<1}h(Y_{s-},x)\tilde{N}(dsdx)|^2\\&=T^{-2r/\al}\E\int_0^{Tt\wedge\tau_{T,K}}\int_{|x|<1}|h(Y_{s-},x)|^2\nu(dx)ds\\&\leq CT^{-2r/\al}\E\int_0^{Tt\wedge\tau_{T,K}}[1+(2+|Y_{s-}|)^{2r-2}]ds\\&\leq CT^{-2r/\al}\E\int_0^{Tt\wedge\tau_{T,K}}[1+(2+KT^{1/\al})^{2r-2}]ds\\&\leq CT^{-2r/\al}Tt[1+(2+KT^{1/\al})^{2r-2}]\\&\leq CT^{-2r/\al}Tt(1+K^{2r-2}T^{\frac{2r-2}{\al}})=Ct(T^{1-\frac{2r}{\al}}+K^{2r-2}T^{1-\frac{2}{\al}}).
	\end{aligned}
\end{equation*} 
Since $1-\frac{2r}{\al}<0$ and $1-\frac{2}{\al}<0,$ we have for each fixed $t\geq 0$ and $K>0,$
\begin{equation}\label{vanishing-L2-J}
	T^{-r/\al}J_{Tt\wedge\tau_{T,K}}\to 0\quad\text{in}\enspace L^2,\quad T\to\infty.
\end{equation}
For each $\eta>0,$ the definition of $\tau_{T,K}$ implies
\begin{align*}
	\{|T^{-r/\al}J_{Tt}|>\eta\}&\subset\{\tau_{T,K}\leq Tt\}\cup\{|T^{-r/\al}J_{Tt\wedge\tau_{T,K}}|>\eta\}\\&\subset\{\sup_{s\leq Tt}|Y_s|\geq KT^{1/\al}\}\cup\{|T^{-r/\al}J_{Tt\wedge\tau_{T,K}}|>\eta\}.
\end{align*}
By Chebyshev inequality,
\begin{equation*}
	\begin{aligned}[b]
		P(|T^{-r/\al}J_{Tt}|>\eta)&\leq P(\sup_{s\leq Tt}|Y_s|\geq KT^{1/\al})+\frac{1}{\eta^2}\E|T^{-r/\al}J_{Tt\wedge\tau_{T,K}}|^2.
	\end{aligned}
\end{equation*}
Letting $T\to\infty$ and then $K\to\infty,$ we establish by \eqref{noname5} and \eqref{vanishing-L2-J} that
\begin{equation*}
	\lim\limits_{T\to\infty}P(|I_{32}^{<1}(t,T)|>\eta)=\lim\limits_{T\to\infty}P(|T^{-r/\al}J_{Tt}|>\eta)=0,
\end{equation*}
and hence also in f.d.d.

\noindent\textbf{Case 2: $|x|\geq 1.$}
Take $\theta\in (\frac{\al}{r},\min\{\al,\frac{\al+p-1}{r-1}\}),$ which does exist. Indeed, recall $r>1,$ so $\frac{\al}{r}<\al.$ Besides, \[r^2-r-\al=r(r-1)-\al<\al(\al-1)-\al=\al(\al-2)<0, \]
so
\[\frac{\al}{r}-\frac{\al+p-1}{r-1}=\frac{\al}{r}-\frac{\al+1-r}{r-1}=\frac{r^2-r-\al}{r(r-1)}<0.\]
Moreover, $\theta>\al/r>1.$
By Lemma \ref{lemma-technique} (iv), when $|y|<1,$
\begin{equation*}
	\int_{|x|\geq 1}|h(y,x)|^\theta\nu(dx)\leq C+C\int_{|x|\geq 1} |y|^\theta|x|^{(r-1)\theta} \nu(dx)\leq C(1+|y|^\theta)\leq C(1+|y|^{r\theta-\al}),
\end{equation*}
where we use
\[(r-1)\theta<(r-1)\frac{\al+p-1}{r-1}=\al+p-1<\al   \] to obtain the convergence of $\int_{|x|\geq 1}|x|^{(r-1)\theta}\nu(dx),$
and
\[(r\theta-\al)-\theta=(r-1)\theta-\al<(r-1)\frac{\al+p-1}{r-1}-\al=p-1<0 \] in the last inequality.
When $|y|\geq 1,$ by Lemma \ref{lemma-technique} (iv),
\begin{equation*}
	\begin{aligned}[b]
		&\quad\int_{|x|\geq 1}|h(y,x)|^\theta\nu(dx)\\&\leq C+C\int_{|x|\geq 1} \min\{ |x|^\theta|y|^{(r-1)\theta},|y|^\theta|x|^{(r-1)\theta} \}\nu(dx)\\&=C+C\int_{1\leq|x|\leq|y|}|x|^\theta |y|^{(r-1)\theta}\nu(dx)+C\int_{|x|\geq |y|} |y|^\theta |x|^{(r-1)\theta}\nu(dx)\\&=C+C|y|^{(r-1)\theta}\int_{1\leq|x|\leq |y|}|x|^{\theta-1-\al}dx+C|y|^\theta\int_{|x|\geq|y|} |x|^{(r-1)\theta-1-\al}dx\\&\leq C+C|y|^{r\theta-\theta}+C|y|^\theta\int_{|x|\geq |y|} |x|^{(r-1)\theta-1-\al}dx.
	\end{aligned}
\end{equation*}
As a consequence of
\[(r-1)\theta-1-\al<(r-1)\frac{\al+p-1}{r-1}-1-\al=(\al+p-1)-1-\al=p-2<-1,\] we have
\begin{equation*}
	\int_{|x|\geq 1}|h(y,x)|^\theta\nu(dx)\leq C+C|y|^{r\theta-\theta}+C|y|^{r\theta-\al}\leq C+C|y|^{r\theta-\theta}.
\end{equation*}
Since
$r\theta-\theta=\theta(r-1)<\frac{\al+p-1}{r-1}(r-1)=\al+p-1,$
\begin{equation*}
	\int_\R\int_{|x|\geq 1}|h(y,x)|^\theta\nu(dx)\mu(dy)\leq\int_\R C+C|y|^{r\theta-\theta}\mu(dy) <\infty.
\end{equation*}

\subsection{Vanishing limit of $I_4$.}

Since \[ \E|I_4(t,T)|^2=T^{-2r/\al}\E\Big|\int_0^{Tt}R(Y_s)ds\Big|^2, \]
$\E|I_4(t,T)|^2=\mathcal{O}(T^{1-\frac{2r}{\al}}),$
provided that we show
\begin{equation}\label{to-be-showed}
	\E\Big|\int_0^{Tt}R(Y_s)ds\Big|^2=\mathcal{O}(T),\quad T\to\infty.
\end{equation}
  Thus by the fact that  $1-\frac{2r}{\al}<0$,
\begin{equation}\label{I4-result}
	\lim\limits_{T\to\infty} \E|I_4(t,T)|^2=0,
\end{equation}
and hence also in f.d.d. sense. The rest of the proof is devoted to showing \eqref{to-be-showed}.

\noindent\textbf{Case 1: $p\geq 1$.}

Note that by a change-of-variable $s=v-u$,
\[ \Big(\int_0^{Tt}R(Y_s)ds\Big)^2=2\int_0^{Tt}\int_u^{Tt}R(Y_v)R(Y_u)dvdu=2\int_0^{Tt}\int_0^{Tt-u}R(Y_{s+u})R(Y_u)dsdu. \]
Taking expectation, utilizing the stationary property of $Y$ and Fubini theorem, we have
\begin{align*}
	&\quad\int_0^{Tt}\int_0^{Tt-u}\E[R(Y_{s+u})R(Y_u)]dsdu\\&=\int_0^{Tt}\int_0^{Tt-u}\E[R(Y_{s})R(Y_0)]dsdu\\&=\int_0^{Tt}\int_0^{Tt-s}\E[R(Y_{s})R(Y_0)]duds\\&=\int_0^{Tt}(Tt-s)\E[R(Y_0)R(Y_s)]ds.
\end{align*}
Combining two equations above,
\begin{equation}\label{noname4}
	\E \Big(\int_0^{Tt}R(Y_s)ds\Big)^2=2\int_0^{Tt}(Tt-s)\E[R(Y_0)R(Y_s)]ds.
\end{equation}
Note that $\mu(R)=0$, a consequence of odd property of $R$ and symmetry of $\mu$, and that $R$ is bounded by Lemma \ref{lemma-boundR}. By \eqref{exp-mixing-2}--\eqref{exp-mixing}, for each $p^\pr\in (0,\al),$
\begin{equation}\label{exp-decay}
		|\E R(Y^y_t)|=|\E R(Y^y_t)-\mu(R)|\leq Ce^{-\lambda t}(1+|y|^{p^\pr}).
\end{equation}
 Let $(P_t)_{t\geq 0}$ be the Markovian semigroup of \eqref{frozen}, then
\[|P_tR(y)|=|\E R(Y^y_t)|\leq Ce^{-\lambda t}(1+|y|^{p^\pr}).  \]
By Markov property,
\begin{equation}\label{markov}
	\E [R(Y_0)R(Y_s)]=\E[R(Y_0)\E(R(Y_s)\mid\mathcal{F}_0)]=\E[R(Y_0)P_sR(Y_0)]=\int R(y)P_sR(y)\mu(dy),
\end{equation}
and since $p^\pr<\al\leq\al+p-1,$ by Lemma \ref{lemma-kulic} (i),
\[|E [R(Y_0)R(Y_s)]|\leq \int |R(y)||P_sR(y)|\mu(dy)\leq Ce^{-\lambda s}\int (1+|y|^{p^\pr})\mu(dy)\leq Ce^{-\lambda s}.\]
Substituting into \eqref{noname4}.
\begin{equation}\label{to-be-showed2}
	\E \Big(\int_0^{Tt}R(Y_s)ds\Big)^2\leq C\int_0^{Tt}(Tt-s) e^{-\lambda s}ds =\mathcal{O}(T).
\end{equation}

\noindent\textbf{Case 2: $0<p<1$.}

First we recall from \eqref{small-p} that in this case, there is necessarily $q=1.$ 
Note that
\begin{align}\label{noname0}
	\E\Big(\int_0^{Tt}R(Y_s)ds\Big)^2&=2\int_0^{Tt}(Tt-s)\E[R(Y_0)R(Y_s)]ds\nonumber\\&\leq 2Tt\int_0^{Tt}|\E[R(Y_0)R(Y_s)]|ds,
\end{align}
 where the first equation is shown by the same argument leading to \eqref{noname4}.
 By \eqref{markov},
 \begin{equation}\label{noname1}
 	\begin{aligned}[b]
 		\quad\int_0^\infty|\E[R(Y_0)R(Y_s)]|ds&=\int_0^\infty \Big|\int R(y)P_sR(y)\mu(dy)\Big|ds\\&\leq \int_0^\infty\int |R(y)|\cdot|P_sR(y)|\mu(dy)ds\\&=\int\int_0^\infty|R(y)|\cdot|P_sR(y)|ds\mu(dy)\\&\leq \|R\|_\infty\int\int_0^\infty|P_sR(y)|ds\mu(dy).
 	\end{aligned}
 \end{equation}
 Note that $\E R(Y^0_s)=0.$ By Lemma \ref{lemma-boundR} (iii),
 \begin{equation}\label{noname2}
 	\begin{aligned}[b]
 		\intot|P_sR(y)|ds&=\intot |\E R(Y^y_s)-\E R(Y^0_s)|ds\\&\leq \intot \E|R(Y^y_s)-R(Y^0_s)|ds\\&\leq C\intot \E|g(Y^y_s)-g(Y^0_s)|ds.
 	\end{aligned}
 \end{equation}
 By \eqref{difference},
 \[|Y^y_t-Y^0_t|-|y|=-\intot |g(Y^y_s)-g(Y^0_s)|ds, \]
 so
 \[\intot |g(Y^y_s)-g(Y^0_s)|ds\leq |y|. \]
 Substituting into \eqref{noname2}, 
 \[\intot|P_sR(y)|ds\leq C|y|.\] Letting $t\to\infty,$
 \[\int_0^\infty|P_sR(y)|ds\leq C|y|.\]
 Substituting into \eqref{noname1} and applying Lemma \ref{lemma-kulic} (i),
 \[\int_0^\infty |\E[R(Y_0)R(Y_s)]|ds\leq \|R\|_\infty\int C|y|\mu(dy)\leq C.\]
 By \eqref{noname0},
 \begin{equation}\label{to-be-showed-3}
 	\E\Big(\int_0^{Tt}R(Y_s)ds\Big)^2\leq CTt.
 \end{equation}
 Combining \eqref{to-be-showed2} and \eqref{to-be-showed-3}, \eqref{to-be-showed} is proved, and this finishes the proof of \eqref{I4-result}.

Substituting \eqref{I1-result}, \eqref{I2-result}, \eqref{I_3-result} and \eqref{I4-result} into \eqref{star},
we obtain
\[(T^{-r/\al}A_{Tt})_{t\geq 0}\to (S_t)_{t\geq 0}\quad\text{in f.d.d.}\]

\section{Brownian Limit}\label{section-Brownian-limit}
In this section we assume that $p\geq q+1-\frac{\al}{2}.$ This particularly implies $r\leq \al/2,$ and $p\geq 2-\frac{\al}{2}> 1,$ so by Lemma \ref{lemma-kulic} (ii), $\mu$ is exponentially mixing. 
\subsection{Normal scaling: $r<\al/2.$}\label{subsection-r<alpha/2}
In this subsection, we exclude the case of $p=q+1-\frac{\al}{2}.$ That is, we always assume $p>q+1-\frac{\al}{2},$ so that $r<\al/2.$ Our goal is to show \eqref{equiv-thm-ii}. 

Let
$\Psi(y):=\int_0^\infty P_sR(y)ds$ which is well-defined by Lemma \ref{lemma-generator}, and  $\Theta:=\hp+\Psi.$
From Lemma \ref{lemma-Holder} (ii),
\[\E\intot\intr|\T(Y_{s-}+x)-\T(Y_{s-})|^2\nu(dx)ds<\infty, \] so
 $\intot\intr\T(Y_{s-}+x)-\T(Y_{s-})\tilde{N}(dsdx)$ is well-defined for each $t\geq0.$
Let
\[M^\Theta(t)=\Theta(Y_t)-\Theta(Y_0)-\intot\mathcal{L}\Theta(Y_s)ds,\enspace I^\Theta(t)=\intot\intr\T(Y_{s-}+x)-\T(Y_{s-})\tilde{N}(dsdx).   \] 
We claim that $M^\T=I^\T.$ Indeed, by Remark \ref{remark-generator},
\begin{equation}\label{new-M-Theta}
	M^\T(t)=\T(Y_t)-\T(Y_0)-\intot f(Y_s)ds.
\end{equation}
Taking expectation, \[\E|M^\T(t)|\leq \E|\T(Y_t)|+\E|\T(Y_0)|+\intot \E|f(Y_s)|ds\leq 2\|\T\|_{L^1(\mu)}+t\|f\|_{L^1(\mu)}<\infty,\]
 and by Markov property, for $0\leq u\leq t<\infty,$
\begin{align*}
	\E(M^\T_t-M^\T_u\mid\F_u)&=\E(\T(Y_t)\mid\F_u)-\T(Y_u)-\int_u^t\E(f(Y_s)\mid\F_u)ds\\&=P_{t-u}\T(Y_u)-\T(Y_u)-\int_u^tP_{s-u}f(Y_u)ds\\&=P_{t-u}\T(Y_u)-\T(Y_u)-\int_0^{t-u}P_sf(Y_u)ds=0,
\end{align*}
where the last equality comes from the generator identity
\begin{equation*}
	P_u\T-\T=\int_0^u P_sf ds\quad\text{in}\enspace L^1(\mu).
\end{equation*}
Since $\T$ is continuous (which is seen from \eqref{psi-increment} and continuity of $\hp$), $M^\T$ is c\`adl\`ag. Therefore, $M^\T$ is a $(\F_t)$ c\`adl\`ag martingale.
On the other hand, $I^\T$ is a c\`adl\`ag $(\mathcal{F}_t)$ local martingales, and so is $M^\T-I^\T.$ Moreover,
\[\Delta M^\T(t)=\T(Y_t)-\T(Y_{t-}),\quad \Delta I^\T(t)=\T(Y_{t-}+\Delta L_t)-\T(Y_{t-})=\Theta(Y_t)-\Theta(Y_{t-}),  \]
the last equality being a consequence of $\Delta L_t=\Delta Y_t,$ thanks to the additive structure of \eqref{frozen}. Therefore, $M^\T-I^\T$ is a continuous $(\mathcal{F}_t)$ local martingale. By the definition of $(\mathcal{F}_t)$ and martingale representation theorem, $M^\T-I^\T$  must be $0.$ In conjunction with \eqref{new-M-Theta},
\[\Theta(Y_t)-\T(Y_0)-\intot f(Y_s)ds=\intot\intr\T(Y_{s-}+x)-\T(Y_{s-})\tilde{N}(dsdx).\]
 Rearranging,
\[\intot f(Y_s)ds=\Theta(Y_t)-\Theta(Y_0)-\intot\intr\T(Y_{s-}+x)-\T(Y_{s-})\tilde{N}(dsdx).\]
Define $\M_t:=\intot\intr\T(Y_{s-}+x)-\T(Y_{s-})\tilde{N}(dsdx),$ we then have
\begin{equation}\label{star2}
	A_{Tt}=\T(Y_{Tt})-\T(Y_0)-\M_{Tt}.
\end{equation}
The definition of $\T$ implies that $\T$ is bounded on compact sets, so by the same argument leading to \eqref{I1-result}, we have
\begin{equation}\label{vanishing-theta}
	\lim\limits_{T\to \infty}	
		\Big(\frac{1}{\sqrt{T}}\Theta(Y_{Tt})\Big)_{t\geq 0}=0\quad\text{in f.d.d,}
\end{equation} 
and
\begin{equation}\label{vanishing-theta2}
	\lim\limits_{T\to \infty}	
	\frac{1}{\sqrt{T}}\Theta(Y_{0})=0\quad\text{ a.s.}
\end{equation} 
Let $\bm^T(t):=\frac{1}{\sqrt{T}}\M_{Tt},$ $\mu^T$ be random measure associated with the jump of $\bm^T$, i.e.
\[\mu^T([0,t]\times B)=\#\{s\in[0,t]:\Delta\bm^T_s\in B\},\quad t\geq 0,\, B\in\mathcal B(\R-\{0\}), \]
and $\nu^T$ be the compensator of $\mu^T$.  Define predictable random field
\[\tT(s,x):=\T(Y_{s-}+x)-\T(Y_{s-}),\] so that $\M_t=\intot\intr\tT(s,x)\tilde{N}(dsdx).$ Since
\begin{align*}
	\mu^T([0,t]\times B)&=\#\{s\in[0,t]:\frac{1}{\sqrt{T}}\Delta\M_{Ts}\in B\}\\&=\#\{s\in[0,Tt]:\frac{1}{\sqrt{T}}\Delta\M_{s}\in B\}\\&=\int_0^{Tt}\intr\mathbf{1}_{\{\frac{1}{\sqrt{T}}\tT(s,x)\in B\}}N(dsdx),
\end{align*}
for all Borel set $B$ with $0\notin B,$ we have
\begin{equation*}
	\nu^T([0,t]\times B)=\int_0^{Tt}\intr\mathbf{1}_{\{\frac{1}{\sqrt{T}}\tT(s,x)\in B\}}\nu(dx)ds.
\end{equation*}
Therefore, for deterministic function $G=G(x),$ 
\begin{equation}\label{integral-vT}
	\begin{aligned}
		\intot\intr G(x)\nu^T(dsdx)=\int_0^{Tt}\intr G(\frac{1}{\sqrt{T}}\tT(s,x))\mathbf{1}_{\{\tT(s,x)\neq0\}}\nu(dx)ds.
	\end{aligned}
\end{equation} Equation \eqref{integral-vT} is checked by testing on indicator function and then an approximation argument.
In particular, setting $G(x)=|x|^2\mathbf{1}_{\{|x|\geq \delta\}}$ for $\delta>0,$ we have
\begin{equation*}
	\begin{aligned}[b]
		&\quad\intot\intr |x|^2\mathbf{1}_{\{|x|\geq \delta\}}\nu^T(dsdx)\\&=\int_0^{Tt}\intr |\frac{1}{\sqrt{T}}\tT(s,x)|^2\mathbf{1}_{\{|\frac{1}{\sqrt{T}}\tT(s,x)|\geq\delta\}}\mathbf{1}_{\{\tT(s,x)\neq 0\}}\nu(dx)ds\\&=\frac{1}{T}\int_0^{Tt}\intr|\T(Y_{s-}+x)-\T(Y_{s-})|^2\mathbf{1}_{\{|\tT(s,x)|\geq\sqrt{T}\delta\}}\mathbf{1}_{\{\tT(s,x)\neq 0\}}\nu(dx)ds\\&\leq \frac{1}{T}\int_0^{Tt}\intr|\T(Y_{s-}+x)-\T(Y_{s-})|^2\mathbf{1}_{\{|\T(Y_{s-}+x)-\T(Y_{s-})|\geq\sqrt{T}\delta\}}\nu(dx)ds.
	\end{aligned}
\end{equation*}
Taking expectation and using the stationarity of $Y$,
\begin{equation}
	\begin{aligned}[b]
		&\quad\E\intot\intr |x|^2\mathbf{1}_{\{|x|\geq \delta\}}\nu^T(dsdx)\\&\leq \frac{1}{T}\int_0^{Tt}\intr\int_\R|\T(y+x)-\T(y)|^2\mathbf{1}_{\{|\T(y+x)-\T(y)|\geq\sqrt{T}\delta\}}\mu(dy)\nu(dx)ds\\&=t\intr\int_\R|\T(y+x)-\T(y)|^2\mathbf{1}_{\{|\T(y+x)-\T(y)|\geq\sqrt{T}\delta\}}\mu(dy)\nu(dx).
	\end{aligned}
\end{equation}
For each $(y,x)\in\R^2$, there is
\[|\T(y+x)-\T(y)|^2\mathbf{1}_{\{|\T(y+x)-\T(y)|\geq\sqrt{T}\delta\}}\to 0,\quad T\to\infty,\] so by Lemma \ref{lemma-Holder} (ii) and dominated convergence theorem,
\begin{equation}
	\E\intot\intr |x|^2\mathbf{1}_{\{|x|\geq \delta\}}\nu^T(dsdx)\to 0,\quad T\to\infty.
\end{equation}
Moreover, by Kunita's result~\cite[Lemma 2.4]{KUN04}, the predictable quadratic variation of $\M_{T\cdot}$ is
\begin{equation}
	\begin{aligned}[b]
		&\quad\la\M_{T\cdot},\M_{T\cdot}\ra_t\\&=\la\int_0^{T\cdot}\intr\T(Y_{s-}+x)-\T(Y_{s-})\tilde{N}(dsdx),\int_0^{T\cdot}\intr\T(Y_{s-}+x)-\T(Y_{s-})\tilde{N}(dsdx)\ra_t\\&=\int_0^{Tt}\intr |\T(Y_{s-}+x)-\T(Y_{s-})|^2\nu(dx)ds,
	\end{aligned}
\end{equation}
so by ergodic theorem \cite[Corollary 10.9]{KAL02},
\begin{equation*}
\begin{aligned}[b]
	&\quad\la\frac{1}{\sqrt{T}}\M_{T\cdot},\frac{1}{\sqrt{T}}\M_{T\cdot}\ra_t\\&=\frac{1}{T}\int_0^{Tt}\intr |\T(Y_{s-}+x)-\T(Y_{s-})|^2\nu(dx)ds\\&\to t\int_\R\intr|\T(y+x)-\T(y)|^2\nu(dx)\mu(dy),\enspace\text{a.s.}
\end{aligned}
\end{equation*}
The last term is well-defined by Lemma \ref{lemma-Holder} (ii). Then $\bm^T\to\bm$ weakly in Skorokhod topology~\cite[Theorem 3.22, Chapter VIII]{JS03}, where $\bm$ is a continuous Gaussian martingale and~\cite[Corollary 3.24, Chapter VIII]{JS03}, 
\begin{equation}\label{def-of-sigma2}
	\la\bm,\bm\ra_t=\sigma^2 t\quad\text{with}\enspace\sigma^2:=\int_\R\intr|\T(y+x)-\T(y)|^2\nu(dx)\mu(dy).
\end{equation}
We claim $\sigma\neq 0.$ Suppose by contradiction that $\sigma=0,$ then
\[\intr |\Theta(y+x)-\Theta(y)|^2\nu(dx)=0,\quad\mu\text{-a.e.}\]
Take $y_0\in\R$ such that
\[\intr|\Theta(y_0+x)-\Theta(y_0)|^2\nu(dx)=0.\] Since $\nu(dx)=|x|^{-1-\al}dx,$ we have
\[|\Theta(y_0+x)-\Theta(y_0)|=0,\quad \text{Lebesgue-a.e.}\,x\in\R.\]
Furthermore, $\Psi$ is continuous by \eqref{psi-increment}, and so is  $\Theta$, which implies 
\[|\Theta(y_0+x)-\Theta(y_0)|=0,\quad\text{for all}\enspace  x\in\R.\]
Consequently, $\Theta$ is a constant and $\LL \Theta=0$ due to the definition of $\LL.$ On the other hand, by Lemma~\ref{lemma-generator}, $\LL \Theta=f.$ Therefore,
$f=0$ in $L^1(\mu).$ Equivalently,
\[\int |f(y)|\mu(dy)=0.\] Given the fact that $f$ only vanishes at $0,$  $\mu=\delta_0,$ i.e., $Y_t=0$ a.s. for each $t\geq 0,$ and since $g(0)=0,$ we have by \eqref{frozen},
\[L_t=Y_t-Y_0-\intot g(Y_s)ds=0\quad \text{a.s,}\] an obvious contradiction.

 By L\'evy characterization theorem,  $\law(\bm)=\law(W),$ where $W$ is a Brownian motion with variance $\sigma^2.$ Note that $-W=W$ in the sense of f.d.d. Combined with \eqref{star2}, \eqref{vanishing-theta} and \eqref{vanishing-theta2}, as $T\to\infty,$
\[\Big(\frac{1}{\sqrt{T}}A_{Tt}\Big)_{t\geq 0}\to (W_t)_{t\geq 0}\quad\text{in f.d.d.}   \]

\subsection{Logarithmic Scaling: $r=\al/2.$}

In this subsection, we focus on the case of $r=\al/2.$ Our goal is to show \eqref{equiv-thm-iii}.

As in \eqref{new-int-f}, 
\begin{equation}\label{star3}
	A_{Tt}=\hp(Y_{Tt})-\hp(Y_0)-M_{Tt}-\int_0^{Tt}R(Y_s)ds.
\end{equation}
By the same argument deriving \eqref{I1-result}--\eqref{I2-result},
\begin{equation}\label{log-vanishing-hp}
	\lim\limits_{T\to\infty}\frac{\hp(Y_{Tt})}{\sqrt{T\log T}}=0,\enspace\lim\limits_{T\to\infty}\frac{\hp(Y_0)}{\sqrt{T\log T}}=0
\end{equation}
in f.d.d. Since $r=\frac{\al}{2},$ we have $p=q+1-r=q+1-\frac{\al}{2}\geq 1,$ which implies that $\mu$ is exponentially mixing. Besides, $R$ is bounded due to Lemma \ref{lemma-boundR}, so by the same argument leading to \eqref{to-be-showed2}, we obtain 
\[ \E\Big(\int_0^{Tt}R(Y_s)ds\Big)^2=\mathcal{O}(T). \]
By Jensen inequality,
\begin{equation*}
	\frac{\int_0^{Tt}R(Y_s)ds}{\sqrt{T\log T}}=\mathcal{O}_{L^1}(\frac{1}{\sqrt{\log T}}),\quad T\to\infty,
\end{equation*}
and hence
\begin{equation}\label{log-vanishing-R}
	\lim\limits_{T\to\infty}\frac{\int_0^{Tt}R(Y_s)ds}{\sqrt{T\log T}}=0\enspace \text{in f.d.d.}
\end{equation}
Now it remains to identify the limit of $M_{Tt}$ as $T\to\infty.$ Decompose $M$ as
\begin{equation}\label{new-M}
	\begin{aligned}[b]
		&\quad M_t\\&=\intot\int_{|x|<1}\hp(Y_{s-}+x)-\hp(Y_{s-})\tilde{N}(dsdx)+\intot\int_{|x|\geq1}\hp(Y_{s-}+x)-\hp(Y_{s-})\tilde{N}(dsdx)\\&=\intot\int_{|x|<1}\hp(Y_{s-}+x)-\hp(Y_{s-})\tilde{N}(dsdx)\\&\quad+\intot\int_{|x|\geq 1}\hp(Y_{s-}+x)-\hp(Y_{s-})-\p(x)\tilde{N}(dsdx)+\intot\int_{|x|\geq 1}\p(x)\tilde{N}(dsdx)\\&=:M^{<1}_t+H_t+Z_t.
	\end{aligned}
\end{equation}
Since
\[\hp^\pr(y)=\p^\pr(y)=-y^{r-1},\quad y\geq 1, \]
and $r=\al/2\in (1/2,1),$ $\hp^\pr$ is bounded on $(1,\infty).$ Besides, $\hp^\pr$ is bounded on compacts, so $\hp^\pr$ is bounded on $[0,\infty).$ The odd property of $\hp$ then implies $\hp^\pr$ is bounded on $\R$, and by mean value theorem,
\begin{equation*}
	\begin{aligned}[b]
		\E|M^{<1}_t|^2&=\E\intot\int_{|x|<1}|\hp(Y_{s-}+x)-\hp(Y_{s-})|^2\nu(dx)ds\\&\leq\E\intot\int_{|x|<1}C|x|^2\nu(dx)ds\\&=Ct\int_{|x|<1}|x|^2\nu(dx).
	\end{aligned}
\end{equation*}
Therefore, for each $t\geq 0,$
\begin{equation*}
	\mathbb{E}\Big|\frac{M^{<1}_{Tt}}{\sqrt{T\log T}}\Big|^2\leq C\int_{|x|<1}|x|^2\nu(dx)\frac{Tt}{T\log T}\leq C\frac{t}{\log T}\to 0,\quad T\to\infty,
\end{equation*}
implying that for each $t\geq 0,$
\begin{equation}\label{log-vanishing-M-leq}
	\lim\limits_{T\to\infty}\frac{M^{<1}_{Tt}}{\sqrt{T\log T}}=0\enspace\text{in f.d.d.}
\end{equation}
Now we turn to $H.$ Recall in Section \ref{subsection-I3},  \[h(y,x)=\hp(y+x)-\hp(y)-\p(x), \] so 
\[H_t=\intot\int_{|x|\geq 1}h(Y_{s-},x)\tilde{N}(dsdx).\]
Since $r=\al/2\in (1/2,1),$ we utilize  Lemma \ref{lemma-technique} (iii) to obtain 
\[|h(y,x)|\leq C+C_r\min\{|x|^r,|y|^r\}. \]
If $|y|<1,$  $|y|<1\leq|x|,$ and then
\begin{equation}
	\begin{aligned}[b]
		\int_{|x|\geq 1}|h(y,x)|^2\nu(dx)&\leq \int_{|x|\geq 1}C+C_r |y|^{2r}\nu(dx)\\&=C+C_r\int_1^\infty|y|^{2r}x^{-1-\al}(dx)\\&\leq C+C_r|y|^{2r}\leq C,
	\end{aligned}
\end{equation}
If $|y|\geq 1,$
\begin{equation}
	\begin{aligned}[b]
			\int_{|x|\geq 1}|h(y,x)|^2\nu(dx)&=\int_{1\leq|x|\leq|y|}|h(y,x)|^2\nu(dx)+\int_{|x|> |y|}|h(y,x)|^2\nu(dx)\\&\leq \int_{1\leq|x|\leq|y|}C+C_r |x|^{2r}\nu(dx)+\int_{|x|> |y|} C+C_r|y|^{2r}\nu(dx)\\&\leq C+C\int_{1}^{|y|}x^{2r-1-\al}(dx)+C\int_{|y|}^\infty|y|^{2r}x^{-1-\al}dx\\&= 2C+C\log|y|,
	\end{aligned}
\end{equation}
where we used $r=\al/2$ in the last equality. By the two inequalities above,
\begin{equation}\label{HTt}
	\begin{aligned}[b]
		\E|H_{Tt}|^2&=\E\Big|\int_0^{Tt}\int_{|x|\geq 1}h(Y_{s-},x)\tilde{N}(dsdx)\Big|^2\\&=\E\int_0^{Tt}\int_{|x|\geq 1}|h(Y_{s-},x)|^2\nu(dx)ds\\&=\int_{\R}\int_0^{Tt}\int_{|x|\geq 1}|h(y,x)|^2\nu(dx)ds\mu(dy)\\&=\int_{|y|<1}\int_0^{Tt}\int_{|x|\geq 1}|h(y,x)|^2\nu(dx)ds\mu(dy)+\int_{|y|\geq 1}\int_0^{Tt}\int_{|x|\geq 1}|h(y,x)|^2\nu(dx)ds\mu(dy)\\&\leq \int_{|y|<1}\int_0^{Tt} C ds\mu(dy)+\int_{|y|\geq1}\int_0^{Tt}C+C\log(|y|)ds\mu(dy)\\&\leq CTt\mu(|y|<1)+CTt\mu(|y|\geq 1)+CTt\int_{|y|\geq 1}(\log|y|)\mu(dy)\\&\leq CTt+CTt\int_{|y|\geq 1}\log |y|\mu(dy).
	\end{aligned}
\end{equation}
By Lemma \ref{lemma-kulic} (i) and the fact that for  $\ell>0,$ $\log(|y|)=o(|y|^\ell)$, $y\to\infty,$ we have $$\int_{|y|\geq 1}\log|y|\mu(dy)<\infty.$$ Substituting into \eqref{HTt}, \[\E|H_{Tt}|^2\leq CTt.\] Furthermore, by Jensen inequality, $\frac{H_{Tt}}{\sqrt{T\log T}}=\mathcal{O}_{L^1}(\frac{1}{\sqrt{\log T}}).$ Therefore
\begin{equation}\label{log-vanishing-H}
	\lim\limits_{T\to\infty} \frac{H_{Tt}}{\sqrt{T\log T}}=0\enspace\text{in f.d.d.}
\end{equation}
At last we explore the limit of $Z_{Tt}.$
 The characteristic function of $Z_t$ is~\cite[(2.9)]{APP09},
\[\E e^{iuZ_t}=e^{t\hat\psi(u)}  \]
where $\hat\psi(u)=\int_\R (e^{iux}-1-iux)\lambda(dx)$ and $\lambda(B):=\nu(\{|x|\geq 1\}\cap \p^{-1}(B))$ for each Borel set $B$ in $\R-\{0\}.$ A direct computation gives 
\begin{equation}\label{def-of-lambda}
	\lambda(dx)=\frac{1}{r^3}|x|^{-3}\mathbf{1}_{\{|x|\geq 1/r\}}dx,
\end{equation}
and hence the exponent of characteristic of $Z_{Tt}/\sqrt{T\log T}$ is
\begin{equation*}
	\begin{aligned}[b]
		&\quad\log\E \exp\Big(iu\frac{Z_{Tt}}{\sqrt{T\log T}}\Big)\\&=Tt\hat\psi(\frac{u}{\sqrt{T\log T}})\\&=Tt\int_\R\exp(i\frac{u}{\sqrt{T\log T}}x)-1-i\frac{ux}{\sqrt{T\log T}}\lambda(dx)\\&=Tt\int_\R\Big[\cos(\frac{u}{\sqrt{T\log T}}x)-1\Big]\lambda(dx),
	\end{aligned}
\end{equation*}
where we used Euler identity the symmetric property of $\lambda$ in the last equality.
But by \eqref{def-of-lambda}, symmetry of $\lambda$ and a change-of-variable,
\begin{equation*}
	\begin{aligned}[b]
		&\quad\int_\R \Big[\cos(\frac{u}{\sqrt{T\log T}}x)-1\Big]\lambda(dx)\\&=\frac{2}{r^3}\int_{1/r}^\infty \Big[\cos(\frac{u}{\sqrt{T\log T}}x)-1\Big]|x|^{-3}dx\\&=\frac{2}{r^3T\log T}\int_{\frac{1}{r\sqrt{T\log T}}}^\infty [\cos(uy)-1]y^{-3}dy.
	\end{aligned}
\end{equation*}
Combining two equations above,
\begin{equation}\label{log-c.f.}
	\log\E \exp\Big(iu\frac{Z_{Tt}}{\sqrt{T\log T}}\Big)=\frac{2t}{r^3\log T}\int_{\frac{1}{r\sqrt{T\log T}}}^\infty [\cos(uy)-1]y^{-3}dy.
\end{equation}
Let $\mathcal{R}(x)=\cos x-1+(x^2/2).$ It is direct to check that $0\leq \mathcal{R}(x)\leq x^4/24,$ and thus for fixed $u\in\R$ and $\delta\in (0,1),$
\begin{equation*}
	\begin{aligned}[b]
		&\quad\Big|\int_\delta^1 [\cos(uy)-1]y^{-3}dy+\frac{u^2}{2}\log\frac{1}{\delta}\Big|\\&=\Big|\int_\delta^1(\mathcal{R}(uy)-\frac{u^2y^2}{2})y^{-3}dy+\frac{u^2}{2}\log\frac{1}{\delta}\Big|\\&=\Big|\int_\delta^1\mathcal{R}(uy)y^{-3}dy\Big|\\&\leq \int_\delta^1 \frac{u^4y}{24}dy\leq \frac{u^4}{48}=\mathcal{O}_u(1),
	\end{aligned}
\end{equation*}
so we have
\[\int_\delta^1 [\cos(uy)-1]y^{-3}dy=-\frac{u^2}{2}\log\frac{1}{\delta}+\mathcal{O}_u(1).\] In addition,
\[\int_1^\infty [\cos(uy)-1]y^{-3}dy=\mathcal{O}_u(1). \]
Combining two equations above,
\begin{equation*}
	\int_\delta^\infty [\cos(uy)-1]y^{-3}dy=-\frac{u^2}{2}\log\frac{1}{\delta}+\mathcal{O}_u(1).
\end{equation*}
Substituting into \eqref{log-c.f.},
\begin{equation*}
	\begin{aligned}[b]
		&\quad\log\E \exp\Big(iu\frac{Z_{Tt}}{\sqrt{T\log T}}\Big)\\&=\frac{2t}{r^3\log T}\Big(-\frac{u^2}{2}\log(r\sqrt{T\log T})+\mathcal{O}_u(1)\Big)\\&=\frac{2t}{r^3\log T}\Big(-\frac{u^2}{2}\log r-\frac{u^2}{4}\log T-\frac{u^2}{4}\log\log T\Big)+\frac{2t}{r^3\log T}\mathcal{O}_u(1)\\&=-\frac{tu^2}{2r^3}+\mathcal{O}_u(\frac{1}{\log T})+\mathcal{O}_u(\frac{\log \log T}{\log T}).
	\end{aligned}
\end{equation*}
Therefore,
\[\lim\limits_{T\to\infty}\log\E \exp\Big(iu\frac{Z_{Tt}}{\sqrt{T\log T}}\Big)=-\frac{tu^2}{2r^3} , \] and hence
\begin{equation*}
	\lim\limits_{T\to\infty}\E \exp\Big(iu\frac{Z_{Tt}}{\sqrt{T\log T}}\Big)=\exp(-\frac{tu^2}{2r^3}).
\end{equation*}
This means, for each $t\geq 0,$
\[\frac{Z_{Tt}}{\sqrt{T\log T}}\to r^{-3/2}B_t\quad\text{in distribution,}\enspace T\to\infty\] for some standard Brownian motion $B.$ But $Z$ is a L\'evy process, so there is necessarily
\begin{equation}\label{log-BM-Z}
	\Big(\frac{Z_{Tt}}{\sqrt{T\log T}}\Big)_{t\geq 0}\to (r^{-3/2}B_t)_{t\geq 0}\enspace\text{in f.d.d.}
\end{equation}
Combining \eqref{new-M}, \eqref{log-vanishing-M-leq}, \eqref{log-vanishing-H} and \eqref{log-BM-Z},
\[ \big(\frac{M_{Tt}}{\sqrt{T\log T}}\big)_{t\geq 0}\to (r^{-3/2}B_t)_{t\geq 0} \enspace\text{in f.d.d.} \]
Combined with \eqref{star3}--\eqref{new-M}, and use the fact that $-B=B$ in the sense of f.d.d, \eqref{equiv-thm-iii} is proved.

\section{Extension to General Initial Distribution}\label{section-extension}
For each $T>1,$ set
\begin{equation*}
	a_T:=\begin{cases}
		T^{r/\al},\quad \al/2<r<\al,\\
		T^{1/2},\quad r<\al/2,\enspace r\neq 0,\\
		(T\log T)^{1/2},\quad r=\al/2.
	\end{cases}
\end{equation*}
We have shown  \eqref{equiv-thm-i},\eqref{equiv-thm-ii} and \eqref{equiv-thm-iii} in Section \ref{section-stable-limit} and Section \ref{section-Brownian-limit}, that is,
\begin{equation}\label{convergence-to-Z}
	\Big(\frac{A_{Tt}}{a_T}\Big)_{t\geq 0}\to (\mathcal{Z}_t)_{t\geq 0} \quad\text{in f.d.d. as $T\to\infty$}
\end{equation}
for each $r<\al$ and $r\neq 0,$ where
\begin{equation*}
		\mathcal{Z}_t:=\begin{cases}
			S_t,\quad \al/2<r<\al,\\
			W_t,\quad r<\al/2,\enspace r\neq 0,\\
			r^{-3/2}B_t,\quad r=\al/2.
		\end{cases}
\end{equation*}
Recall that $A_t:=\intot (Y_s)^qds$ and $Y$ is the solution  to \eqref{frozen} with initial value $Y_0\sim\mu$\,.  In this section, our aim is to generalize \eqref{convergence-to-Z} to
\begin{equation}\label{convergence-to-Z-2}
	\Big(\frac{A^\eta_{Tt}}{a_T}\Big)_{t\geq 0}\to (\mathcal{Z}_t)_{t\geq 0} \quad\text{in f.d.d.}
\end{equation}
where $A^\eta_t:=\intot (Y^\eta_s)^qds,$ and $Y^\eta$ is the solution of \eqref{frozen} with initial value $\eta$ which is independent of $(L_t)_{t\geq 0}.$

\subsection{Deterministic Initial Value}
In this subsection, we assume $\eta\in\R$ is deterministic.

\subsubsection{Exponentially mixing case: $p\geq 1.$}

Fix $0<t_1<...<t_m,$ and choose $H:\R^m\to\R$ such that
\begin{equation}\label{H-condition}
	|H(x)|\leq 1,\enspace |H(x)-H(z)|\leq \sum_{i=1}^{m}|x_i-z_i|,\quad x,z\in\R^m.
\end{equation}
 For fixed $h>0,$ and each $T>\frac{h}{t_1},$ define random vectors
\[U^\eta_{T,h}:=\Big(\frac{1}{a_T}\int_h^{Tt_j}(Y^\eta_s)^qds\Big)_{j=1}^m,\quad V^\eta_T:=\Big(\frac{1}{a_T}\int_0^{Tt_j}(Y^\eta_s)^qds\Big)_{j=1}^m,\]
so that 
\[V^\eta_T-U^\eta_{T,h}=\Big(\frac{1}{a_T}\int_0^h(Y^\eta_s)^qds\Big)_{j=1}^m=\frac{A^\eta_h}{a_T}(1,1,...,1).\]
By \eqref{H-condition},
\[|H(V^\eta_T)-H(U^\eta_{T,h})|\leq 2\wedge\frac{m}{a_T}|A^\eta_h|.\]
Taking expectation and applying bounded convergence theorem, for fixed $h>0,$
\begin{equation}\label{noname8}
	\E|H(V^\eta_T)-H(U^\eta_{T,h})|\to 0,\quad T\to\infty.
\end{equation}
By stationarity,
\begin{equation*}
	\E|A_{Tt_j}-A_{Tt_j-h}|\leq \E\int_{Tt_j-h}^{Tt_j}|Y_s|^qds=\int_{Tt_j-h}^{Tt_j}\E|Y_s|^qds\leq h\int|y|^q\mu(dy),
\end{equation*}
and an application of \eqref{H-condition} yields for each $h>0,$
\begin{align}\label{noname9}
	\Big|\E H((\frac{A_{Tt_j-h}}{a_T})_{j=1}^m)-\E H((\frac{A_{Tt_j}}{a_T})_{j=1}^m)\Big|\nonumber&\leq \E\Big|H((\frac{A_{Tt_j-h}}{a_T})_{j=1}^m)-H((\frac{A_{Tt_j}}{a_T})_{j=1}^m)\Big|\nonumber\\&\leq \E\sum_{j=1}^{m}\Big|\frac{A_{Tt_j-h}-A_{Tt_j}}{a_T}\Big|\nonumber\\&\leq \frac{m}{a_T}h\int|y|^q\mu(dy)\to 0,\quad T\to\infty.
\end{align}
Define $F_{T,h}:\R\to\R$ by 
\[F_{T,h}(y):=\E\Big[H((\frac{1}{a_T}\int_0^{Tt_j-h}(Y^y_s)^qds)_{j=1}^m)\Big],\] and the path functional $\phi_{T,h}: D([0,\infty))\to\R$ by
\[\phi_{T,h}(w):= H\Big((\frac{1}{a_T}\int_0^{Tt_j-h}(w(s))^qds)_{j=1}^m\Big),\quad w\in D([0,\infty)),\] so that 
\begin{equation}\label{path-equation}
	F_{T,h}(y)=\E\phi_{T,h}(Y^y)=\int\phi_{T,h}(w)P_y(dw).
\end{equation}
Here $P_y$ is the law of $Y^y$ on $D([0,\infty)).$ A change-of-variable gives \[U^\eta_{T,h}=\Big(\frac{1}{a_T}\int_0^{Tt_j-h}(Y^\eta_{s+h})^qds\Big)_{j=1}^m,\] so
\[H(U^\eta_{T,h})=H\Big((\frac{1}{a_T}\int_0^{Tt_j-h}(Y^\eta_{s+h})^qds)_{j=1}^m\Big)=\phi_{T,h}((Y^\eta_{s+h})_{s\geq 0})=\phi_{T,h}(\theta_h\circ Y^\eta).\] Here $\theta_h$ is the shift operator defined by $(\theta_h w)(s)=w(s+h).$
 By Markov property and \eqref{path-equation}, \[\E[H(U^\eta_{T,h})\mid\F^{Y^\eta}_h]=\E[\phi_{T,h}(\theta_h\circ Y^\eta)\mid\F^{Y^\eta}_h]=(\E_y\phi_{T,h})|_{y=Y^\eta_h}=\Big(\int\phi_{T,h}dP_y\Big)\Big|_{y=Y^\eta_h}=F_{T,h}(Y^\eta_h).\] Taking expectation
\[\E H(U^\eta_{T,h})=\E F_{T,h}(Y^\eta_h)=P_h F_{T,h}(\eta).\]
Also note
\[\E \Big[H((\frac{A_{Tt_j-h}}{a_T}\Big)_{j=1}^m)\mid\F_0\Big]=\E\Big[H((\frac{1}{a_T}\int_0^{Tt_j-h}(Y_s)^qds)_{j=1}^m)\mid\F_0\Big]=\E[\phi_{T,h}(Y)\mid\F_0]. \] By Markov property and \eqref{path-equation},
\[\E \Big[H((\frac{A_{Tt_j-h}}{a_T}\Big)_{j=1}^m)\mid\F_0\Big]=(\E_y\phi_{T,h})|_{y=Y_0}=\Big(\int\phi_{T,h}dP_y\Big)_{y=Y_0}=F_{T,h}(Y_0).\] Taking expectation,
\[\E H\Big(\Big(\frac{A_{Tt_j-h}}{a_T}\Big)_{j=1}^m\Big)=\mu(F_{T,h}),\]
and notice that $\|F_{T,h}\|_\infty\leq 1,$ a consequence of \eqref{H-condition}, and thus we have
\begin{equation*}
	\Big|\E H(U^\eta_{T,h})-\E H((\frac{A_{Tt_j-h}}{a_T})_{j=1}^m)\Big|\leq |P_hF_{T,h}(\eta)-\mu(F_{T,h})|\leq d_\eta(h),
\end{equation*}
where $d_\eta(h):=\sup_{\|\vp\|\leq 1}|P_h\vp(\eta)-\mu(\vp)|.$ Since $p\geq 1,$ by Lemma \ref{lemma-kulic} (ii), $d_\eta(h)\to 0$ as $h\to \infty$ for each deterministic $\eta\in\R,$ and hence
\begin{equation}\label{noname10}
	\lim\limits_{h\to \infty}\limsup_{T\to\infty}\Big|\E H(U^\eta_{T,h})-\E H((\frac{A_{Tt_j-h}}{a_T})_{j=1}^m)\Big|=0.
\end{equation}
Notice that
\begin{align*}
	&\quad\Big|\E H(V^\eta_T)-\E H((\frac{A_{Tt_j}}{a_T})_{j=1}^m)\Big|\\&\leq |\E H(V^\eta_T)-\E H(U^\eta_{T,h})|+\Big|\E H(U^\eta_{T,h})-\E H((\frac{A_{Tt_j-h}}{a_T})_{j=1}^m)\Big|\\&\quad+\Big|\E H((\frac{A_{Tt_j-h}}{a_T})_{j=1}^m)-\E H((\frac{A_{Tt_j}}{a_T})_{j=1}^m)\Big|.
\end{align*}
The left-hand-side of the inequality above is independent of $h.$ Therefore, letting $T\to\infty$ first and then $h\to \infty,$ we have by the definition of $V^\eta,$ \eqref{noname8}, \eqref{noname9} and \eqref{noname10} that
\begin{equation}\label{noname11}
	\limsup_{T\to\infty}\Big|\E  H((\frac{A^\eta_{Tt_j}}{a_T})_{j=1}^m)-\E H((\frac{A_{Tt_j}}{a_T})_{j=1}^m)\Big|=\limsup_{T\to\infty}\Big|\E H(V^\eta_T)-\E H((\frac{A_{Tt_j}}{a_T})_{j=1}^m)\Big|=0.
\end{equation}
By \eqref{convergence-to-Z},
\[\lim\limits_{T\to\infty}\E H\Big((\frac{1}{a_T}A_{Tt_j})_{j=1}^m\Big)=\E H((\mathcal{Z}_{t_j})_{j=1}^m).\]
Combined with \eqref{noname11},
\[\lim\limits_{T\to\infty}\E H\Big((\frac{1}{a_T}A^\eta_{Tt_j})_{j=1}^m\Big)=\E H((\mathcal{Z}_{t_j})_{j=1}^m),\]
for each $H$ satisfying \eqref{H-condition}, and hence for each bounded Lipschitz function $H,$ which implies \nobreak\cite[Theorem 11.3.3]{DUD02}
\[\Big(\frac{1}{a_T}A^\eta_{Tt_j}\Big)_{j=1}^m\to (\mathcal Z_{t_j})_{j=1}^m\quad\text{weakly}.\] \qed

\subsubsection{Slowly mixing case: $0<p<1.$}

Recall from \eqref{small-p} that in this case there is necessarily $q=1,$ so \eqref{convergence-to-Z-2} reduces to
\begin{equation}\label{convergence-to-Z-3}
	\Big(\frac{A^\eta_{Tt}}{T^{(2-p)/\al}}\Big)_{t\geq 0}\to (\mathcal{Z}_t)_{t\geq 0} \quad\text{in f.d.d.}
\end{equation}
 The additive structure of \eqref{frozen} gives
\[\frac{d}{dt}|Y^\eta_t-Y_t|=-|g(Y^\eta_t)-g(Y_t)|,\quad \text{a.e.}\enspace t\geq0,\] so
\begin{equation}\label{integral}
	|Y^\eta_t-Y_t|+\intot|g(Y^\eta_s)-g(Y_s)|ds=|Y^\eta_0-Y_0|.
\end{equation}
Note also that for each $M>0,$
\begin{equation}\label{g-increment}
	|g(u)-g(v)|\geq pM^{p-1}|u-v|,\quad 0\leq |u|,|v|\leq M.
\end{equation}
Indeed, $g$ is absolutely continuous with $g^\pr(x)=-p|x|^{p-1}$ for $x\neq0$, so for $u<v,$
\[|g(v)-g(u)|=p\Big|\int_u^v |x|^{p-1}dx\Big|\geq pM^{p-1}(v-u),\] where we used $p<1$ in the last inequality, which proves \eqref{g-increment}.

For each $u>0,$ set $M_u:=1+\sup_{0\leq s\leq u}(|Y^\eta_s|\vee|Y_s|).$ For each $s\in [0,u],$ applying \eqref{g-increment} with $u=Y^\eta_s, v=Y_s$ and $M=M_u,$ we obtain
\[|Y^\eta_s-Y_s|\leq \frac{1}{pM_u^{p-1}}|g(Y^\eta_s)-g(Y_s)|.\]
Integrating and applying \eqref{integral},
\begin{equation}\label{int-Y-increment}
	\int_0^u|Y^\eta_s-Y_s|ds\leq \frac{M_u^{1-p}}{p}\int_0^u|g(Y^\eta_s)-g(Y_s)|ds\leq \frac{M_u^{1-p}}{p}|Y^\eta_0-Y_0|.
\end{equation}
By Lemma \ref{lemma-Y-L}, we have
\[\sup_{0\leq s \leq u}|Y^\eta_s|\leq |Y^\eta_0|+2\sup_{0\leq s \leq u}|L_s|,\quad \sup_{0\leq s \leq u}|Y_s|\leq |Y_0|+2\sup_{0\leq s \leq u}|L_s|,\] so
\[M_u\leq 1+|Y^\eta_0|+|Y_0|+2\sup_{0\leq s \leq u}|L_s|.\]
Dividing by $T^{1/\al},$
\[\frac{M_u}{T^{1/\al}}\leq \frac{1+|Y^\eta_0|+|Y_0|}{T^{1/\al}}+\frac{2\sup_{0\leq s \leq u}|L_s|}{T^{1/\al}}.\] For each fixed $K>0,$ taking $u=TK$ and using the scaling property of $L,$
\begin{align*}
	P(\frac{M_{TK}}{T^{1/\al}}\geq R)&\leq P(1+|Y^\eta_0|+|Y_0|\geq \frac{T^{1/\al}R}{2})+P(T^{-1/\al}\sup_{0\leq s \leq TK}|L_s|\geq \frac{R}{4})\\&\leq P(1+|Y^\eta_0|+|Y_0|\geq \frac{R}{2})+P(T^{-1/\al}\sup_{0\leq s \leq TK}|L_s|\geq \frac{R}{4})\\&=P(1+|Y^\eta_0|+|Y_0|\geq \frac{R}{2})+P(\sup_{0\leq s \leq K}|L_s|\geq \frac{R}{4})
\end{align*}
for each $R>0$ and $T\geq 1.$ The right-hand-side of the inequality above is independent of $T,$ so
\[\lim_{R\to\infty}\sup_{T\geq 1}P(\frac{M_{TK}}{T^{1/\al}}\geq R)=0.\]
This means precisely that 
\begin{equation}\label{tight-MTK}
	\Big(\frac{M_{TK}}{T^{1/\al}}\Big)_{T\geq 1}\quad\text{is tight for each}\enspace K>0.
\end{equation}
Recall $q=1,$ so
\[\frac{|A^\eta_{Tt}-A_{Tt}|}{T^{(2-p)/\al}}\leq \frac{\int_0^{Tt}|Y^\eta_s-Y_s|ds}{T^{(2-p)/\al}}.\]
Taking supremum on $[0,K]$ and applying \eqref{int-Y-increment}, 
\begin{equation}\label{noname7}
	\sup_{0\leq t\leq K}\frac{|A^\eta_{Tt}-A_{Tt}|}{T^{(2-p)/\al}}\leq\frac{\int_0^{TK}|Y^\eta_s-Y_s|ds}{T^{(2-p)/\al}}\leq \frac{|Y^\eta_0-Y_0|}{p}\Big(\frac{M_{TK}}{T^{1/\al}}\Big)^{1-p}T^{-1/\al}.
\end{equation}
 Fix $\rho>0$ and $\delta>0.$ For each $K>0,$ by \eqref{tight-MTK}, there is $B>0$ such that
\begin{equation}\label{>B}
	P(\frac{M_{TK}}{T^{1/\al}}>B)<\rho/2,\quad T\geq 1.
\end{equation} 
Since $|Y^\eta_0-Y_0|$ is almost surely finite, there is $R>0$ such that
\begin{equation}\label{>R}
	P(|Y^\eta_0-Y_0|>R)<\rho/2.
\end{equation}
Take $T_0$ sufficiently large such that $\frac{RB^{1-p}}{p}T_0^{-1/\al}\leq\delta.$ For each $T\geq T_0,$ by \eqref{noname7},
\[\sup_{0\leq t\leq K}\frac{|A^\eta_{Tt}-A_{Tt}|}{T^{(2-p)/\al}}\leq\frac{RB^{1-p}}{p}T_0^{-1/\al}\leq\delta\] on the event $\{|Y^\eta_0-Y_0|\leq R\}\cap\{\frac{M_{TK}}{T^{1/\al}}\leq B\}.$ Therefore, we have
\[\{\sup_{0\leq t\leq K}\frac{|A^\eta_{Tt}-A_{Tt}|}{T^{(2-p)/\al}}>\delta\}\subset \{|Y^\eta_0-Y_0|> R\}\cup\{\frac{M_{TK}}{T^{1/\al}}>B\},\]
and by \eqref{>B}--\eqref{>R},
\begin{align*}
	 P(\sup_{0\leq t\leq K}\frac{|A^\eta_{Tt}-A_{Tt}|}{T^{(2-p)/\al}}>\delta)&\leq P(|Y^\eta_0-Y_0|> R)+ P(\frac{M_{TK}}{T^{1/\al}}>B)\\&\leq \rho/2+\rho/2=\rho,\quad T\geq T_0.
\end{align*}
Since $\delta$ and $\rho$ are arbitrary, we have for each $K>0,$
\[\lim_{T\to\infty}\sup_{0\leq t\leq K}\frac{|A^\eta_{Tt}-A_{Tt}|}{T^{(2-p)/\al}}=0\quad\text{in probability,}\] which implies \eqref{convergence-to-Z-3}.\qed

\subsection{Random Independent Initial Value}
In this subsection, we assume $\eta$ is a real random variable independent of $(L_t)_{t\geq 0}.$

Set $V^y_T:=(\frac{1}{a_T}A^y_{Tt_j})_{j=1}^m$ for each deterministic $y\in\R,$ then $V^y_T\to (\mathcal Z_{t_j})_{j=1}^m$ weakly due to the deterministic result, so 
\begin{equation}\label{deterministic}
	\E H(V^y_T)\to \E H((\mathcal Z_{t_j})_{j=1}^m)
\end{equation}
for all $H\in C_b(\R^m;\R).$ For random variable $\eta$ independent of $L,$ let $s_t(y,\ell)$ be the solution of~\eqref{s_t} evaluated at $t$ for each $y\in\R$ and $\ell\in D([0,\infty)).$ Then
\[ V^y_T=(\frac{1}{a_T}\int_0^{Tt_j} f(Y^y_r)dr)_{j=1}^m=(\frac{1}{a_T}\int_0^{Tt_j} f(s_r(y,L))dr)_{j=1}^m=G_T(y,L),  \]
where \[G_T(y,\ell):=(\frac{1}{a_T}\int_0^{Tt_j} f(s_r(y,\ell))dr)_{j=1}^m.\] 
Since $\eta$ is independent of $L,$
\[\E[H(V^\eta_T)\mid\sigma(\eta)]=\E[H(G_T(\eta,L))\mid\sigma(\eta)]=(\E[H(G_T(y,L))])|_{y=\eta}=(\E H(V^y_T))|_{y=\eta}. \]
Denote the distribution of $\eta$ as $P_\eta.$ Taking expectation,
\[\E H(V^\eta_T)=\E[(\E H(V^y_T))|_{y=\eta}]=\int \E H(V^y_T)P_\eta(dy). \] By bounded convergence theorem and \eqref{deterministic},
\[\lim\limits_{T\to\infty}\E H(V^\eta_T)=\int\E H((\mathcal Z_{t_j})_{j=1}^m) dP_\eta=\E H((\mathcal Z_{t_j})_{j=1}^m) \] for each $H\in C_b(\R^m;\R).$\qed

\appendix
\section{Appendix}\label{appendix}

\begin{lemma}\label{lemma-wellposed}
	For each $p>0$ and $\al\in (1,2),$ equation \eqref{frozen} (and hence \eqref{slow-fast}) admits a unique strong solution.
\end{lemma}

\begin{proof}
	When $p\geq 1,$ \eqref{frozen} is an additive SDE with locally Lipschitz drift, so it admits a local unique strong solution. The monotonicity and dissipativity give pathwise uniqueness and non-explosion on $(0,\infty).$
	
	 When $p\in (0,1),$ we apply the  Fournier's results~\cite[Proposition 2 and Theorem 4]{FOU13} to obtain the weak existence of solution of \eqref{frozen} and the pathwise uniqueness. Then equation \eqref{frozen} admits a unique strong solution~ by the Yamada--Watanabe theorem \cite[Theorem 2]{BLP15}.
\end{proof}

We introduce the semigroup of solution of \eqref{frozen}. For each $t\geq0,$ $y\in\R,$ Borel set $B\subset\R,$ and bounded measurable function $h$, define
\begin{equation*}
	P_t(y,B):=P(Y^y_t\in B),\quad P_th(y):=\E h(Y^y_t).
\end{equation*}

	\begin{lemma}\label{lemma-kulic}
		(i) Fix $p>0$. The equation \eqref{frozen} admits a unique invariant measure $\mu$, and for $\ell\geq 0,$
		\[\int |y|^\ell\mu(dy)<\infty, \]
		if and only if $\ell<\al+p-1.$
		
		(ii) If $p>1$,  there are constants $C>0$ and $\lambda>0,$ such that for all $t\geq 0$ and $y\in\R$,
		\begin{equation}\label{exp-mixing-2}
			\sup_{\|\vp\|_\infty\leq 1}|\E\varphi(Y^y_t)-\mu(\varphi)|\leq Ce^{-\lambda t},\quad y\in\R.
		\end{equation}
		Here $\mu(\vp):=\int\vp d\mu.$
		
		 If $p=1,$ then for each $p^\pr\in (0,\al),$ there are constants $C_{p^\pr}>0$ and $\lambda>0$ such that for all  $t\geq 0$ and $y\in\R,$
		\[\sup_{\|\vp/V_{p^\pr}\|_\infty\leq 1}|\E\vp(Y^y_t)-\mu(\vp)|\leq C_{p^\pr}e^{-\lambda t}V_{p^\pr}(y),\quad y\in\R,\]
		where $V_{p^\pr}(y):=1+|y|^{p^\pr}.$ In particular, for $p\geq 1$ and $p^\pr\in (0,\al),$ if $\|\vp\|_\infty<\infty,$ there is a constant $C$ depending on $\vp,$ such that
		\begin{equation}\label{exp-mixing}
			|\E\vp(Y^y_t)-\mu(\vp)|\leq Ce^{-\lambda t}(1+|y|^{p^\pr}),\quad y\in\R.
		\end{equation}

	\end{lemma}
	\begin{proof}
 We prove this lemma by using results of Zhang and Zhang~\cite{ZZ23}. To do this we need to identify the quantities $\ell_1,$ $\ell_2$ and $r$ required in \cite{ZZ23}. Since \eqref{frozen} is additive, it is immediate that $\ell_2=0.$ Note that when $p\in (0,1),$ $g$ is $p$-H\"older continuous, so $\ell_1(y)=\|g\|_p$ with $\|g\|_p$ the $p$-H\"older semi-norm of $g.$ When $p\geq1$, by mean value theorem, for fixed $\gamma\in (0,1]$ and $|h|\leq 1,$
		\[|g(y+h)-g(y)|=|g^\pr(y+\theta h)|\cdot|h|=p|y+\theta h|^{p-1}\cdot|h|\leq p(|y|+1)^{p-1}|h|\leq C_p(1+|y|^{p-1})|h|^\gamma,\] so we choose $\ell_1(y)=C_p(1+|y|^{p-1}).$ Taking $\ep_0=1,$ \cite[(5)]{ZZ23} becomes
		\begin{equation}\label{case1}
			-|y|^{p+1}+\|g\|_p|y|+q^\pr|y|^{2-\al}\leq -c_0|y|^{2+r}+c_1,\quad |y|\geq 1
		\end{equation} if $p\in (0,1),$ and
		\begin{equation}\label{case2}
			-|y|^{p+1}+C_p|y|+C_p|y|^p+q^\pr|y|^{2-\al}\leq -c_0|y|^{2+r}+c_1,\quad |y|\geq 1
		\end{equation}
		if $p\in [1,\infty).$ Here $q^\pr$ is a constant depending only on $\al$ and $p$~\cite[(22)]{ZZ23}.
		
		In $p\in (0,1)$ case, take $c_0=1/2,$ $r=p-1,$ and \[c_1:=\sup\limits_{|y|\geq 1}[-\frac{1}{2}|y|^{p+1}+\|g\|_p|y|+q^\pr|y|^{2-\al}].\] Since $p+1>1>2-\al,$ $c_1<\infty,$ and \eqref{case1} holds.
		
		 In $p\in [1,\infty)$ case, take $c_0=1/2,$ $r=p-1,$ and 
		\[c_1:=\sup_{|y|\geq 1} [-\frac{1}{2}|y|^{p+1}+C_p|y|+C_p|y|^p+q^\pr|y|^{2-\al}].\] Since $p+1>p\geq 1>2-\al,$ $c_1<\infty,$ and \eqref{case2} holds.

		Next we prove (i). The existence and uniqueness of invariant measure $\mu$ is established by \cite[Theorem 1.2]{ZZ23}.
		
		Now we show the finite $\ell$-th moment of $\mu$ for $\ell\in [0,\al+p-1).$ Consider the integral equation
		\[Y^y_t=y+\intot g(Y^y_s)ds+\intot\int_{|x|\geq 1}x N(dsdx)+\intot\int_{|x|<1}x\tilde{N}(dsdx). \]
		It is direct to find that
		\begin{equation*}
			\int_{|x|>1}|x|^{p^\pr}\nu(dx)=2\int_1^\infty x^{p^\pr-1-\al}dx<\infty
		\end{equation*}
		if and only if $p^\pr<\al$,
		and that
		\begin{equation*}
		  |Y_t|\geq 1\implies g(Y_t)Y_t=-|Y_t|^{1+p}\leq -|Y_t|^{1+\ka}
		\end{equation*}
		if and only if $\ka\leq p.$ We additionally require $p^\pr>1$ and $\ka>0,$ so that $p^\pr+\ka>1.$ Now ``set of assumptions ($\mathbf{A}$)'' in \cite{KP21} is satisfied.
		
		 Taking $\sigma=0$ in \cite[Corollary 2.7 (ii)]{KP21},
		\[\frac{1}{t}\E\Big(\intot |Y^y_s|^{p^\pr+\ka-1}ds\mid\mathcal{F}_0\Big)\leq\frac{C}{t}|y|^{p^\pr}+\frac{C}{t}+C,\quad t>0.  \]
		Taking expectation,
		\begin{equation}\label{KP21-coro}
			\frac{1}{t}\E\intot|Y^y_s|^{p^\pr+\ka-1}ds\leq\frac{C}{t}|y|^{p^\pr}+\frac{C}{t}+C,\quad t>0.
		\end{equation}
		Set \[\mu^y_t(A):=\frac{1}{t}\intot P_s(y,A)ds,\quad A\in\mathcal{B}(\R),t>0,\] so that for each bounded Borel function $F,$
		\begin{equation}\label{time-average}
			\int_\R F(z)\mu^y_t(dz)=\frac{1}{t}\intot\int_\R F(z)P_s(y,dz)ds=\frac{1}{t}\intot\E F(Y^y_s)ds=\frac{1}{t}\intot P_sF(y)ds,
		\end{equation}
		where the first equality is verified by testing on indicator function, simple function and then bounded Borel function by dominated convergence theorem.
		Therefore, by \eqref{time-average}, \eqref{KP21-coro} and monotone convergence theorem,
		\begin{equation}\label{p+k-1-moment}
			\begin{aligned}[b]
				\int |z|^{p^\pr+\ka-1}\mu^y_t(dz)&=\lim\limits_{N\to\infty}\int(|z|^{p^\pr+\ka-1}\wedge N)\mu^y_t(dz)\\&=\lim\limits_{N\to\infty}\frac{1}{t}\intot\E[|Y^y_s|^{p^\pr+\ka-1}\wedge N]ds\\&\leq\frac{C}{t}|y|^{p^\pr}+\frac{C}{t}+C.
			\end{aligned}
		\end{equation}
		For each $R>0,$ set $B_R:=\{z\in\R:|z|\leq R\}.$ By Markov inequality,
		\[\mu^y_t(B_R^c)\leq \frac{\int|z|^{p^\pr+\ka-1}\mu^y_t(dz)}{R^{p^\pr+\ka-1}}\leq\frac{1}{R^{p^\pr+\ka-1}}(\frac{C}{t}|y|^{p^\pr}+\frac{C}{t}+C).\]
		Therefore, $(\mu^y_t)_{t\geq1}$ is tight in $\R$, and by Prokhorov theorem, there is a subsequence $(t_k)_{k\in\N}$ with $\lim\limits_{k\to\infty}t_k=\infty$ such that $\mu^y_{t_k}\to\mu^y$ weakly.  Therefore, for each $N\in\N,$ by \eqref{p+k-1-moment},
		\[\int(|z|^{p^\pr+\ka-1}\wedge N)\mu^y(dz)=\lim\limits_{k\to\infty}\int(|z|^{p^\pr+\ka-1}\wedge N)\mu^y_{t_k}(dz)\leq\lim\limits_{k\to\infty} \frac{C}{t_k}|y|^{p^\pr}+\frac{C}{t_k}+C=C. \]
		Let $N\to\infty,$ by monotone convergence theorem,
		\begin{equation}\label{mu-x-bound}
			\int |z|^{p^\pr+\ka-1}\mu^y(dz)\leq C.
		\end{equation}
		Now we prove that $\mu^y=\mu$ for each $y\in\R.$ First note that 
		\[Y^y_t-Y^z_t=y-z+\intot g(Y^y_s)-g(Y^z_s)ds,\quad Y^y_0-Y^z_0=y-z, \]
		and by definition of $g$ and absolute continuity of $t\mapsto |Y^y_t-Y^z_t|,$ 
		\begin{equation}\label{difference}
			\frac{d}{dt}|Y^y_t-Y^z_t|=\sgn(Y^y_t-Y^z_t)\cdot (g(Y^y_t)-g(Y^z_t))=-|g(Y^y_t)-g(Y^z_t)|\leq 0,\enspace\text{a.e.}\enspace t\geq 0.
		\end{equation}
		implying that 
		\begin{equation}\label{Yy-contraction}
			|Y^y_t-Y^z_t|\leq |y-z|,\enspace t\geq 0.
		\end{equation} For $t\geq 0$ and $F\in C_b(\R),$ we claim that $P_tF\in C_b(\R).$  Indeed, for $(y_n)\subset \R$ with $y_n\to y,$ by~\eqref{Yy-contraction}, $Y^{y_n}_t\to Y^y_t$ almost surely, and continuity of $F$ implies $F(Y^{y_n}_t)\to F(Y^y_t)$ almost surely. By bounded convergence theorem,
		\[|P_tF(y_n)-P_tF(y)|=|\E F(Y^{y_n}_t)-\E F(Y^y_t)|\leq \E|F(Y^{y_n}_t)-F(Y^y_t)|\to 0.\]
		 This shows $P_tF$ is continuous. The boundedness of $P_tF$ follows directly from the definition of $(P_t),$ so $P_tF\in C_b(\R).$    By \eqref{time-average}, for each $v\geq0,$
		\begin{align*}
			\Big|\int P_vFd\mu^y_t-\int Fd\mu^y_t\Big|&=\Big|\frac{1}{t}\intot P_s(P_vF)(y)ds-\frac{1}{t}\intot P_uF(y)du\Big|\\&=\Big|\frac{1}{t}\intot P_{s+v}F(y)ds-\frac{1}{t}\intot P_uF(y)du\Big|\\&=\Big|\frac{1}{t}\int_v^{t+v}P_uF(y)du-\frac{1}{t}\intot P_u F(y)du\Big|\\&=\Big|\frac{1}{t}\int_t^{t+v}P_uF(y)du-\frac{1}{t}\int_0^v P_u F(y)du\Big|\\&\leq \frac{1}{t}\int_t^{t+v}|P_uF(y)|du+\frac{1}{t}\int_0^v|P_uF(y)|du\\&\leq \frac{1}{t}\int_t^{t+v}\|P_uF\|_\infty du+\frac{1}{t}\int_0^v\|P_uF\|_\infty du\\&\leq \frac{1}{t}\int_t^{t+v}\|F\|_\infty du+\frac{1}{t}\int_0^v\|F\|_\infty du\leq \frac{2v\|F\|_\infty}{t}.
		\end{align*}
		Since $\mu^y_{t_k}\to\mu^y$ weakly, we pass to a subsequence $(t_k)$ to obtain
		\[\Big|\int P_v Fd\mu^y-\int Fd\mu^y\Big|=0,\] for each $v\geq 0.$
		so $\mu^y$  is an invariant measure of $(P_t).$ By the uniqueness, $\mu^y=\mu.$ By \eqref{mu-x-bound},
		\begin{equation*}
			\int |z|^{p^\pr+\ka-1}\mu(dz)<\infty.
		\end{equation*}
		Since $p^\pr$ and $\ka$ can be arbitrarily close to $\al$ and $p$, respectively, we conclude that 
		\[\int |z|^{\ell}\mu(dz)<\infty\] for all $0\leq\ell<\al+p-1.$  
		
		Now we prove that, if $\ell\geq \al+p-1,$ then $\int|z|^\ell\mu(dz)=\infty.$ Indeed, for test function $\vp\in C_c^\infty((1,2))$ with $\vp\geq 0$ and $\vp\not\equiv 0$ set $\vp_R(y):=\vp(y/R)$ for each $R>0,$ then $\vp_R$ is supported on $[R,2R].$ Note also $\|\vp_R^\pr\|_\infty=R^{-1}\|\vp_R^\pr\|_\infty,$ so
		\begin{equation}\label{noname12}
			|\int g(y)\vp_R^\pr(y)\mu(dy)|\leq \int_R^{2R} |g(y)|\cdot|\vp_R^\pr(y)|\mu(dy)\leq \|\vp_R^\pr\|_\infty\int_R^{2R}|g(y)|\mu(dy)\leq CR^{p-1}\mu([R,2R]).
		\end{equation}
		By dominated convergence theorem,
		\[\int\mathcal J\vp(y/R)\mu(dy)\to\mathcal J\vp(0)>0,\quad R\to\infty,\] where $\mathcal Jh(y):=\intr [h(y+x)-h(y)-xh^\pr(y)]\nu(dx),$
		so there is $R_0>0$ such that 
		\begin{equation}\label{noname13}
			\int\mathcal J\vp(y/R)\mu(dy)\geq  \frac{\mathcal J\vp(0)}{2}>0,\quad R\geq R_0.
		\end{equation}
		Applying the  It\^o's formula to $\vp_R(Y_t),$
		\begin{equation}\label{vpR-ito}
			\begin{aligned}[b]
				&\quad\vp_R(Y_t)-\vp_R(Y_0)\\&=\intot\vp_R^\pr(Y_s) g(Y_s)ds+\intot\intr \vp_R(Y_{s-}+x)-\vp_R(Y_{s-})\tilde{N}(dsdx)+\intot\mathcal J\vp_R(Y_s)ds.
			\end{aligned}
		\end{equation} 
		 The second term on the right-hand-side of \eqref{vpR-ito} is a martingale, and since $Y$ is stationary, taking expectation on \eqref{vpR-ito} yields
		\[\E\Big[\intot\vp_R^\pr(Y_s)g(Y_s)ds+\intot\mathcal J\vp_R(Y_s)ds\Big]=0.\] The left-hand-side is precisely $t[\int\vp_R^\pr(y)g(y)\mu(dy)+\int\mathcal J\vp_R(y)\mu(dy)],$ and also notice that $\mathcal J\vp_R(y)=R^{-\al}\mathcal J\vp(y/R)$ due to the scaling property of $\nu,$ which gives
		\[\int\vp_R^\pr(y)g(y)\mu(dy)=-\int\mathcal J\vp_R(y)\mu(dy)=-R^{-\al}\int\mathcal J\vp(y/R)\mu(dy).\]  Substituting into \eqref{noname12} and applying \eqref{noname13}, for all $R\geq R_0,$
		\[\mu([R,2R])\geq C\frac{1}{R^{p-1}}|\int g(y)\vp_R^\pr(y)\mu(dy)|=C\frac{1}{R^{p-1}}R^{-\al}|\int\mathcal J\vp(y/R)\mu(dy)|\geq CR^{-(\al+p-1)}.\]
		Choose $n_0\in\N$ such that $3^{n_0}\geq R_0.$ Taking disjoint intervals $[3^n,2\cdot 3^n]$ for $n\geq n_0,$ we have
		\[\int_0^\infty y^\ell \mu(dy)\geq \sum_{n=n_0}^\infty\int_{3^n}^{2\cdot 3^n}y^\ell\mu(dy)\geq\sum_{n=n_0}^\infty  3^{n\ell} \mu([3^n,2\cdot 3^n])\geq \sum_{n=n_0}^\infty C\cdot 3^{n(\ell-(\al+p-1))}. \] If $\ell\geq \al+p-1,$ the series is necessarily divergent to infinity. This completes the proof of (i). 
		
		Next we show (ii). Recall that $r$ in \cite{ZZ23} can be chosen as $p-1.$ So when $p>1,$ $r$ can be chosen as a strictly positive number, and when $p=1,$ the biggest admissible $r$ is $0.$ An application of \cite[Theorem 1.2]{ZZ23} gives (ii).
	\end{proof}
	\begin{lemma}\label{lemma-semigroup}
		$(P_t)_{t\geq 0}$ is uniquely extended to $L^1(\mu),$ and $(P_t)_{t\geq 0}$ is a strongly continuous contraction semigroup on $L^1(\mu).$
	\end{lemma}
	\begin{proof}
		Note that
		\[\|P_th\|_{L^1(\mu)}=\int|P_th(z)|\mu(dz)=\int|\E h(Y^z_t)|\mu(dz)\leq\int \E|h(Y^z_t)|\mu(dz)=\int [P_t|h|](z)\mu(dz).\] Since $\mu$ is $(P_t)$-invariant,
		\[\int [P_t|h|](z)\mu(dz)=\int|h(z)|\mu(dz)=\|h\|_{L^1(\mu)}.\] Combining two equations above,
		\[\|P_th\|_{L^1(\mu)}\leq \|h\|_{L^1(\mu)},\quad h\in B_b(\R).\]
		Therefore, a density argument shows that for each $t\geq 0,$ $P_t$ can be uniquely extended to $L^1(\mu),$ and that $(P_t)$ is a contraction semigroup on $L^1(\mu).$ 
		
		If $h\in C_b(\R),$
		\begin{equation*}
			\|P_th-h\|_{L^1(\mu)}=\int |\E h(Y^y_t)-h(y)|\mu(dy)\leq\int \E|h(Y^y_t)-h(y)|\mu(dy).
		\end{equation*}
		 Since $h\in C_b(\R)$ and $Y$ is c\`adl\`ag, by bounded convergence theorem,
		\[\lim_{t\to 0}\|P_th-h\|_{L^1(\mu)}=0.\] Since $(P_t)$ is a contraction, a density argument yields
		\[\lim_{t\to 0}\|P_th-h\|_{L^1(\mu)}=0,\quad h\in L^1(\mu).\]
		Therefore, $(P_t)$ is a strongly continuous contraction semigroup on $L^1(\mu).$
	\end{proof}
	\begin{remark}
		If $r\leq \al/2$ (so that $p>1$), then $f,R,\hp\in L^1(\mu),$ and hence lie in the domain of $(P_t).$ Indeed, $R$ is globally bounded, so it is in $L^1(\mu).$ The fact that $f\in L^1(\mu)$ follows directly from Lemma~\ref{lemma-kulic} (i).  $\hp$ has bounded derivative as a consequence of $r\leq \al/2<1,$ so $\hp$ is globally Lipschitz, and hence is of linear growth, and an application of Lemma~\ref{lemma-kulic} (i) yields $\hp\in L^1(\mu).$
	\end{remark}
	Denote the domain of generator $\LL$ of $(P_t)_{t\geq 0}$ by
	\[D(\LL):=\Big\{h\in L^1(\mu):\lim\limits_{t\to0}\frac{P_th-h}{t}\enspace\text{in}\enspace L^1(\mu)\enspace\text{exists.}\Big\} \]
If $p>1,$ we set $$\Psi(y):=\int_0^\infty P_sR(y)ds,\quad \text{and}\quad \Theta:=\hp+\Psi.$$
	The integral is absolutely convergent. Indeed, by Lemma \ref{lemma-kulic} (ii) and the fact that $\int Rd\mu=0,$ we have
	\[|P_sR(y)|=|\E R(Y^y_s)|\leq Ce^{-\lambda s},\] so
	\[\int_0^\infty |P_sR(y)|ds\leq \int_0^\infty Ce^{-\lambda s}ds<\infty.\] 
	
	The next lemma claims that $\T$ solves a Poisson equation in $L^1(\mu).$
	\begin{lemma}\label{lemma-generator} Suppose that $r<\al/2$ (so that $p>1$), then
		\begin{equation}
			\Theta\in D(\LL)\quad\text{and}\quad	\LL\Theta=f\enspace\text{in}\enspace L^1(\mu).
		\end{equation}
	\end{lemma}

	\begin{proof}
		By Fubini theorem,
		\[P_t\Psi(y)=\int\Psi(z)P_t(y,dz)=\int\int_0^\infty P_sR(z)dsP_t(y,dz)=\int_0^\infty\int P_sR(z)P_t(y,dz)ds. \]
		By semigroup property,
		\[\int_0^\infty\int P_sR(z)P_t(y,dz)ds=\int_0^\infty P_sP_tR(y)ds=\int_0^\infty P_{s+t}R(y)ds=\int_t^\infty P_sR(y)ds. \]
		Combining two equations above, 
		\[P_t\Psi(y)=\int_t^\infty P_sR(y)ds,\] so
		\begin{equation*}
			P_t\Psi(y)-\Psi(y)=\int_t^\infty P_sR(y)ds-\int_0^\infty P_sR(y)ds=-\intot P_sR(y)ds.
		\end{equation*}
		Therefore,
		\begin{equation*}
			\begin{aligned}[b]
				\Big\|\frac{P_t\Psi-\Psi}{t}+R\Big\|_{L^1(\mu)}&=\int\Big|\frac{-\intot P_sR(y)ds}{t}+\frac{\intot R(y)ds}{t}\Big|\mu(dy)\\&=\frac{1}{t}\int\Big|\intot P_sR(y)-R(y)ds\Big|\mu(dy)\\&\leq \frac{1}{t}\int\int_0^t|P_sR(y)-R(y)|ds\mu(dy)\\&\leq\frac{1}{t}\intot\int|P_sR(y)-R(y)|\mu(dy)ds\\&=\frac{1}{t}\intot\|P_sR-R\|_{L^1(\mu)}ds.
			\end{aligned}
		\end{equation*}
		By mean value theorem for integrals and strong continuity of $(P_t),$
		\[\lim\limits_{t\to 0}\Big\|\frac{P_t\Psi-\Psi}{t}+R\Big\|_{L^1(\mu)}=0,\] and this means
		\begin{equation}\label{L-Psi}
			\Psi\in D(\LL)\quad\text{and}\quad \LL\Psi=-R.
		\end{equation}
		Also, from \eqref{new-int-f} one sees that
		\[\hp(Y_t)-\hp(Y_0)=\intot (f+R)(Y_s)ds+M_t\] with $M_t=\intot\intr\hp(Y_{s-}+x)-\hp(Y_{s-})\tilde{N}(dsdx).$ The process $M$ is a true martingale. Indeed, for $r\leq \al/2<1,$ $\hp^\pr$ is bounded, so
		\[|\hp(Y_{s-}+x)-\hp(Y_{s-})|\leq C|x|,\]
		and then \[\intot\int_{|x|<1}|\hp(Y_{s-}+x)-\hp(Y_{s-})|^2\nu(dx)ds\leq Ct\int_{|x|<1}|x|^2\nu(dx)<\infty,\] 
		If $r<0,$ $\hp$ is bounded, so
		\[|\hp(Y_{s-}+x)-\hp(Y_{s-})|\leq C,\] from which
		\[\intot\int_{|x|\geq 1}|\hp(Y_{s-}+x)-\hp(Y_{s-})|^2\nu(dx)ds\leq Ct\nu(|x|\geq 1)<\infty.\] For $r\in (0,\al/2),$ $\hp$ is $r$-H\"older continuous, so
		\[|\hp(Y_{s-}+x)-\hp(Y_{s-})|\leq C|x|^r,\] which implies
		\[\intot\int_{|x|\geq 1}|\hp(Y_{s-}+x)-\hp(Y_{s-})|^2\nu(dx)ds\leq C\intot\int_{|x|\geq 1}|x|^{2r}\nu(dx)ds=Ct\int_{|x|\geq 1}|x|^{2r-1-\al}dx<\infty,\] because $r<\al/2.$ Therefore, $M$ is  square-integrable, and in particular a true martingale.
		
		 Conditioning on $\sigma(Y_0),$
		\begin{equation}\label{condition-Ito}
			\E[\hp(Y_t)\mid\sigma(Y_0)]-\hp(Y_0)=\E[\hp(Y_t)-\hp(Y_0)\mid\sigma(Y_0)]=\E[\intot (f+R)(Y_s)ds\mid\sigma(Y_0)].
		\end{equation} Set $H_t(y,\ell):=h(s_t(y,\ell)),$ where $s_t$ is the solution of
		\begin{equation}\label{s_t}
			y_t=y+\intot g(y_s)ds+\ell_t.
		\end{equation} evaluated at time $t,$
		 so that \[H_t(y,L)=h(Y^y_t)\quad\text{and}\quad H_t(Y_0,L)=h(Y_t).\] Since $L$ is independent of $\sigma(Y_0),$ for each bounded measurable function $h,$
		\begin{equation}\label{noname3}
			\E[h(Y_t)\mid\sigma(Y_0)]=\E[H_t(Y_0,L)\mid\sigma(Y_0)]=(\E H_t(y,L))|_{y=Y_0}=(\E h(Y^y_t))|_{y=Y_0}=P_th(Y_0).
		\end{equation}
		For $h\in L^1(\mu),$ take $h_N:=(h\wedge N)\vee(-N),$ so  $h_N\to h$ in $ L^1(\mu),$ and by strong continuity of $(P_t),$ $P_th_N\to P_th$ in $L^1(\mu).$ Since $\law(Y_0)=\law(Y_t)=\mu,$ we have
		\[\E|\E[h_N(Y_t)\mid\sigma(Y_0)]-\E[h(Y_t)\mid\sigma(Y_0)]|\leq \E|h_N(Y_t)-h(Y_t)|=\|h_N-h\|_{L^1(\mu)}\to 0,\] and
		\[\E|P_th_N(Y_0)-P_th(Y_0)|=\|P_th_N-P_th\|_{L^1(\mu)}\to 0.\]
		Passing to a subsequence, we find that \eqref{noname3} is valid for each $h\in L^1(\mu),$
		so
		\[\E[\hp(Y_t)\mid\sigma(Y_0)]=P_t\hp(Y_0),\quad \E[(f+R)(Y_t)\mid\sigma(Y_0)]=P_t(f+R)(Y_0). \]
		By \eqref{condition-Ito} and conditional Fubini theorem,
		\begin{align*}
			P_t\hp(Y_0)-\hp(Y_0)&=\E[\hp(Y_t)\mid\sigma(Y_0)]-\hp(Y_0)\\&=\E[\intot (f+R)(Y_s)ds\mid\sigma(Y_0)]\\&=\intot\E[(f+R)(Y_s)\mid\sigma(Y_0)]ds\\&=\intot P_s(f+R)(Y_0)ds.
		\end{align*}
		Since $\law(Y_0)=\mu,$ this means precisely
		\[P_t\hp-\hp=\intot P_s(f+R)ds\quad\text{in}\enspace L^1(\mu).\] 
		By strong continuity of $(P_t)$ and mean value theorem for integral,
		\begin{align*}
			\Big\|\frac{P_t\hp-\hp}{t}-(f+R)\Big\|_{L^1(\mu)}&=\Big\|\frac{\intot P_s(f+R)ds}{t}-\frac{\intot (f+R) ds}{t}\Big\|_{L^1(\mu)}\\&\leq \frac{1}{t}\intot\|P_s(f+R)-(f+R)\|_{L^1(\mu)}ds\to 0,\quad t\to 0.
		\end{align*}
		Therefore,
		\[\hp\in D(\LL)\quad\text{and}\quad \mathcal{L}\hp=f+R.\]
		Combined with \eqref{L-Psi}, we have
		\begin{equation*}
			\LL\Theta=\LL\hp+\LL\Psi=(f+R)-R=f.
		\end{equation*}
	\end{proof}
		\begin{remark}\label{remark-generator}
		Since $\LL\Theta=f$ in $L^1(\mu),$ and $\law(Y_s)=\mu$ for all $s\geq 0,$ we have
		\[\E|\LL\Theta(Y_s)-f(Y_s)|=\int_\R|\LL\Theta(y)-f(y)|\mu(dy)=0.\] By Fubini theorem, for each $n\in\N,$
		\[\E\int_0^n|\LL\Theta(Y_s)-f(Y_s)|ds=\int_0^n\E|\LL\Theta(Y_s)-f(Y_s)|ds=0,\] which implies
		\[\int_0^n|\LL\Theta(Y_s)-f(Y_s)|ds=0,\quad\text{a.s.}\] Therefore, the processes $\int_0^\cdot\LL\Theta(Y_s)ds$ and $\int_0^\cdot f(Y_s)ds$ are indistinguishable.
	\end{remark}
		\begin{lemma}\label{lemma-technique}
		Define $h(y,x):=\hp(y+x)-\hp(y)-\p(x).$ We have the following results.
		
			(i) Set $x_+:=\max\{x,0\}.$ For each $r\in\R-\{0\},$ 	\[ |\hp^\pr(z)|\leq  C(1+|z|)^{(r-1)_+},\quad z\in\R. \]
			
			(ii) For each  $r\in\R-\{0\},$ \[ |h(y,x)|\leq C_r(|x|(2+|y|)^{(r-1)_+}+|x|^r),\quad 0<|x|<1,\,y\in\R.\]
			
			(iii) If $0<r<1,$
			\[|h(y,x)|\leq  C+C_r\min\{|x|^r,|y|^r\},\quad x\neq0,\,y\in\R. \]
			
			(iv) If $1\leq r\leq 2$, with the convention $0^0:=1,$
			\[|h(y,x)|\leq C+C_r\min\{|x||y|^{r-1},|y||x|^{r-1}\},\quad x\neq0,\,y\in\R.\]
		\end{lemma}
		
		\begin{proof}
			(i) Since $\hp\in C^3(\R)$, we have
			\begin{equation}\label{small-z}
				|\hp^\pr(z)|\leq C\leq C(1+|z|)^{(r-1)_+},\quad |z|\leq 1.
			\end{equation}
			
			If $|z|>1$, then $\hp^\pr(z)=\p^\pr(z)=-|z|^{r-1}.$ Therefore, if $r\geq 1$,
			\[ |\hp^\pr(z)|= |z|^{r-1}\leq (1+|z|)^{r-1}= (1+|z|)^{(r-1)_+},\quad |z|\geq 1. \]
			If $r<1$,
			\[ |\hp^\pr(z)|= |z|^{r-1}\leq 1=(1+|z|)^{(r-1)_+} ,\quad |z|\geq 1.\]
			Combining two inequalities above, for all $r\in\R-\{0\},$
			\[ |\hp^\pr(z)|\leq (1+|z|)^{(r-1)_+},\quad |z|\geq 1. \]
			Combining with \eqref{small-z}, the proof is finished.
			
			(ii) 
			By mean value theorem, \[ \hp(y+x)-\hp(y)=x\int_0^1\hp^\pr(y+\theta x)d\theta, \]
			so by Lemma \ref{lemma-technique} (i)
			\begin{equation*}
				\begin{aligned}[b]
					|\hp(y+x)-\hp(y)|&\leq |x|\int_0^1|\hp^\pr(y+\theta x)|d\theta\\&\leq C |x|\int_0^1 (1+|y+\theta x|)^{(r-1)_+}d\theta.
				\end{aligned}
			\end{equation*}
			Since $|x|<1$, we continue the equation above as
			\begin{equation*}
				\begin{aligned}[b]
					|\hp(y+x)-\hp(y)|&\leq C|x|\int_0^1 (2+|y|)^{(r-1)_+}d\theta\\&=C|x|(2+|y|)^{(r-1)_+}.
				\end{aligned}
			\end{equation*}
			Note also that $|\p(x)|=\frac{1}{|r|}|x|^r=C_r|x|^r.$
			Consequently
			\begin{equation*}
				|h(y,x)|\leq C_r(|x|(2+|y|)^{(r-1)_+}+|x|^r).
			\end{equation*}
			
			(iii)
			When $0<r<1,$ $\p$ is $r$-H\"older continuous. The definition of $\p$ gives $|\p(x)|\leq C_r|x|^r.$ Therefore,
			\[|\p(x+y)-\p(y)-\p(x)|\leq |\p(x+y)-\p(y)|+|\p(x)|\leq C_r|x|^r. \]
			Replacing $(x,y)$ in equation above by $(y,x)$, we obtain
			\[|\p(x+y)-\p(y)-\p(x)|\leq |\p(x+y)-\p(x)|+|\p(y)|\leq C_r|y|^r. \]
			Therefore,
			\[|\p(x+y)-\p(y)-\p(x)|\leq C_r\min\{|x|^r,|y|^r\}, \]
			and since $\bp$ is bounded by Remark \ref{remark-odd-bound} (ii),
			\begin{equation*}
				\begin{aligned}[b]
					|h(y,x)|&=|\hp(y+x)-\hp(y)-\p(x)|\\&\leq|\p(y+x)-\p(y)-\p(x)|+|\bar{\p}(y+x)-\bar{\p}(y)|\\&\leq C+C_r\min\{|x|^r,|y|^r\}.
				\end{aligned}
			\end{equation*}
			
			(iv)
	 When $r\geq 1,$	by mean value theorem,
		\[\p(y+x)-\p(y)=x\int_0^1\p^\pr(y+sx)ds, \] and then we have
		\begin{equation}\label{mean-value}
			|\p(y+x)-\p(y)|\leq |x|\int_0^1|\p^\pr(y+sx)|ds\leq|x|\int_0^1|y+sx|^{r-1}ds.
		\end{equation}
		
		\noindent\textbf{Case 1: $|x|\leq|y|.$}
		
		We have \[|y+sx|\leq|y|+s|x|\leq 2|y|,\quad 0\leq s\leq 1,\]
		and because $r-1\geq 0,$
		$|y+sx|^{r-1}\leq (2|y|)^{r-1}\leq C_r|y|^{r-1},$ substituting which into \eqref{mean-value} gives
		\begin{equation*}
			|\p(y+x)-\p(y)|\leq |x|\int_0^1 C_r|y|^{r-1}ds\leq C_r|x||y|^{r-1}.
		\end{equation*}
		Also, given that $r-1\geq 0,$ \[|\p(x)|=\frac{1}{r}|x|^{r}=C_r|x||x|^{r-1}\leq C_r|x||y|^{r-1}.\]
		Combining two inequalities above,
		\begin{equation}\label{small-x-case}
			|\p(y+x)-\p(y)-\p(x)|\leq C_r|x||y|^{r-1}=C_r\min\{|x||y|^{r-1},|y||x|^{r-1}\},
		\end{equation}
		where we used $r\leq 2$ in the last equality.
		
		\noindent\textbf{Case 2: $|x|> |y|.$}
		
		Since $|x|>	|y|,$ $|y|<|x|.$ Applying \eqref{small-x-case} after interchanging $x$ and $y$, we have
		\begin{equation}\label{big-x-case}
				|\p(y+x)-\p(y)-\p(x)|\leq C_r\min\{|x||y|^{r-1},|y||x|^{r-1}\}.
		\end{equation}
		The proof is finished by combining \eqref{small-x-case}--\eqref{big-x-case} and the fact that $\bp$ is bounded.

		\end{proof}

		\begin{lemma}\label{lemma-boundR}
		(i)	If $r\in(-\infty,0)\cup(0,\al),$ $R$ is bounded.
		
		(ii) If $r\in (-\infty,0)\cup (0,\alpha/2],$ $R$ is globally Lipschitz.
		
		(iii) If $r\in (\al/2,\al)$ and $0<p<1$ (so that $q=1$), 
		\[|R(v)-R(u)|\leq C_p|g(v)-g(u)|,\quad u<v.\]
		\end{lemma}
		\begin{proof}
		 (i)	Recall that  \[ R(y)=\bar{g}(y)+\int_0^\infty \hp(y+x)+\hp(y-x)-2\hp(y)\nu(dx)=:\bar{g}(y)+K(y),\] and that $\bar{g}$ is a bounded function by Remark \ref{remark-odd-bound} (ii), so it suffices to show $K$ is bounded. By definition of $\hp$, we have
			\[ \hp(y)=\p(y)=-\frac{1}{r}y^r,\quad y\geq 1. \]
			Therefore, 
			\begin{equation}\label{hpprpr}
				\hp^{\pr}(y)=-y^{r-1},\enspace \hp^{\pr\pr}(y)=-(r-1)y^{r-2},\quad y\geq 1.
			\end{equation}
			
			\noindent\textbf{Case 1: $r\in (0,\al).$}
			\begin{equation}\label{order-of-hp}
				|\hp^\pr(y)|\leq |y|^{r-1},\enspace |\hp^{\pr\pr}(y)|\leq C_r|y|^{r-2},\quad y\geq 1.
			\end{equation}
			By symmetric property of $\hp^\pr$ and $\hp^{\pr\pr}$ on $(-\infty,-1]\cup [1,\infty)$, \eqref{order-of-hp} is valid for all $y$ with $|y|\geq 1.$ By \cite[Lemma B.1]{KP19},
			\begin{equation*}
				|K(y)|\leq C(1+|y|)^{r-\al}\leq C,\quad y\in\R,
			\end{equation*}
			where we used $r<\al$ in the last inequality.
			
			\noindent\textbf{Case 2: $r\in(-\infty,0).$}
			
			By \eqref{hpprpr}, and the fact that $\hp^\pr$ and $\hp^{\pr\pr}$ are bounded on compacts, they are bounded in $\R.$
			By Taylor formula,
			\[\hp(y+x)+\hp(y-x)-2\hp(y)=x^2\int_0^1 (1-s)[\hp^{\pr\pr}(y+sx)+\hp^{\pr\pr}(y-sx)]ds,\]
			and the boundedness of $\hp^{\pr\pr}$ gives
			\[|\hp(y+x)+\hp(y-x)-2\hp(y)|\leq Cx^2.\] Since $r<0,$ we also have $\hp$ is bounded, which yields
			\[|\hp(y+x)+\hp(y-x)-2\hp(y)|\leq C(x^2\wedge1),\]
			from which 
			\[|K(y)|\leq C\int_0^\infty |x|^2\wedge 1\,\nu(dx)\leq C.\]
			
			 (ii) Since $\bar g$ is $C^1$ and is compactly supported, its derivative must be globally bounded, so $\bar g$ is globally Lipschitz. It remains to show $K$ is Lipschitz, and it suffices to show $K^\pr$ is bounded. Since $r\leq \al/2<1,$ 
			\[\|\hp^\pr\|_\infty\vee\|\hp^{\pr\pr\pr}\|_\infty<\infty.\] By Taylor formula,
			\[\hp^{\pr}(y+x)+\hp^\pr(y-x)-2\hp^\pr(y)=x^2\int_0^1 (1-s)[\hp^{\pr\pr\pr}(y+sx)+\hp^{\pr\pr\pr}(y-sx)]ds, \] and it follows that
			\[|\hp^{\pr}(y+x)+\hp^\pr(y-x)-2\hp^\pr(y)|\leq C\] and
			\[|\hp^{\pr}(y+x)+\hp^\pr(y-x)-2\hp^\pr(y)|\leq Cx^2\int_0^1 |1-s|ds\leq Cx^2.\]
			Therefore,
			\[|K^\pr(y)|\leq \int_0^\infty |\hp^\pr(y+x)+\hp^\pr(y-x)-2\hp^\pr(y)|\nu(dx)\leq \int_0^\infty C(x^2\wedge1)\nu(dx)\leq C.\]
			This finishes the proof.
			
			(iii) We first show
			\begin{equation}\label{R-derivative}
				|R^\pr(y)|\leq C|y|^{p-1},\quad y\neq 0.
			\end{equation}
			Indeed, since $q=1,$ by definition of $\bar g,$
			\[\bar g^\pr(y)=g^\pr(y)\hp^\pr(y)+g(y)\hp^{\pr\pr}(y)-1,\quad y\neq 0.\] Therefore, for $0<|y|\leq 1,$ by the boundedness of $\hp^\pr$ and $\hp^{\pr\pr}$ on $[-1,1],$ we have
			\begin{equation}\label{bar g}
				\begin{aligned}[b]
					|\bar g^\pr(y)|&\leq |g^\pr(y)|\cdot|\hp^\pr(y)|+|g(y)|\cdot|\hp^{\pr\pr}(y)|+1\\&\leq C|g^\pr(y)|+C|g(y)|+1\\&\leq C|y|^{p-1}+|y|^p+1\leq C|y|^{p-1}.
				\end{aligned}
			\end{equation}
			For $|y|>1,$ $\bar g^\pr(y)=0$ by Remark \ref{remark-odd-bound} (ii), so \eqref{bar g} holds for all $y\neq 0.$
			
			Now we deal with $K.$ Differentiating on the definition of $K,$ 
			\[K^\pr(y)=\int_0^\infty \hp^\pr(y+x)+\hp^\pr(y-x)-2\hp^\pr(y)\nu(dx).\]
			A direct computation gives
			\[|\hp^{\pr\pr}(y)|=(1-p)|y|^{-p},\enspace |\hp^{\pr\pr\pr}(y)|\leq p(1-p)|y|^{-p-1},\quad |y|\geq 1.\] By \cite[Lemma B.1]{KP19} (with $\ep=1$ in the reference),
			\begin{equation}\label{K-pr}
				|K^\pr(y)|\leq C(1+|y|)^{1-p-\al}.
			\end{equation} Because $p<1,$ this also implies $K^\pr$ is bounded. Therefore, if $0<|y|\leq 1,$ $|K^\pr(y)|\leq C\leq C|y|^{p-1},$ and by \eqref{bar g},
			\[|R^\pr(y)|\leq |\bar g^\pr(y)|+|K^\pr(y)|\leq C|y|^{p-1},\quad 0<|y|<1. \]
			Note that $2-p=q+1-p=r<\al,$ we have $\al+p>2,$ so $1-p-\al<-1<p-1.$ Therefore, \eqref{K-pr} implies 
			\[|K^\pr(y)|\leq C(1+|y|)^{p-1},\quad y\in\R,\] and if $|y|>1,$ the equation above becomes \[|K^\pr(y)|\leq C|y|^{p-1},\quad |y|>1.\]   Combined with \eqref{bar g},
			\[|R^\pr(y)|\leq |\bar g^\pr(y)|+|K^\pr(y)|\leq C|y|^{p-1},\quad |y|>1. \]
			This finishes the proof of \eqref{R-derivative}.  Since 
			\[R(y)=\bar g(y)+K(y)=g(y)\hp^\pr(y)-y+K(y),\] it is absolutely continuous on each bounded interval. 
			 By \eqref{R-derivative}, for $u,v\in\R$ with $u<v,$
			\[|R(v)-R(u)|\leq \int_u^v |R^\pr(z)|dz\leq C\int_u^v|z|^{p-1}dz.\]
			But \[\int_u^v |z|^{p-1}dz=\frac{1}{p}(|v|^p\sgn(v)-|u|^p\sgn(u))=\frac{1}{p}|g(v)-g(u)|,\] and combining two equations above gives
			\[|R(v)-R(u)|\leq C_p|g(v)-g(u)|.\]

			\end{proof}

			\begin{lemma}\label{lemma-Holder}
				Suppose that $r<\al/2$ and $r\neq 0$ (so that $p>1$). 
				
				(i)	For each $\eta\in (\al/2,1),$ 
				\[|\T(y+x)-\T(y)|\leq C_\eta |x|^\eta,\quad x\in [-1,1]. \]
				
				(ii) \[\sup_{y\in\R}\intr|\T(y+x)-\T(y)|^2\nu(dx)<\infty.\]
				In particular, $(x,y)\mapsto|\T(y+x)-\T(y)|^2\in L^1(\nu\otimes\mu).$
			\end{lemma}

			\begin{proof}
				(i) By Lemma \ref{lemma-boundR} (ii), $R$ is globally Lipschitz, then \eqref{Yy-contraction} implies,
				\[|P_t R(y)-P_t R(z)|=|\E [R(Y^y_t)-R(Y^z_t)]|\leq\E|[R(Y^y_t)-R(Y^z_t)]|\leq\|R\|_{\operatorname{Lip}}\E|Y^y_t-Y^z_t|\leq C|y-z|. \]
				But recalling that $|P_tR(y)|\leq Ce^{-\lambda t},$ which implies $|P_tR(y)-P_tR(z)|\leq C\min\{e^{-\lambda t},|y-z|\},$ and hence for $y,z\in\R$ with $0<{\tiny }|y-z|\leq 1$ and $\eta\in (0,1),$
				\begin{equation*}
					\begin{aligned}[b]
						|\Psi(y)-\Psi(z)|&\leq \int_0^\infty|P_tR(y)-P_tR(z)|dt\\&\leq C\int_0^\infty\min\{e^{-\lambda t},|y-z|\}dt\\&=C\int_0^{\frac{1}{\lambda}\log\frac{1}{|y-z|}}|y-z|dt+C\int_{\frac{1}{\lambda}\log\frac{1}{|y-z|}}^\infty e^{-\lambda t}dt\\&=\frac{C}{\lambda}|y-z|(1+\log(\frac{1}{|y-z|}))\\&\leq C_\eta |y-z|^\eta,
					\end{aligned}
				\end{equation*}
				where we use $a(1+\log(1/a))\leq C_\eta a^\eta,\enspace 0<a\leq 1$ in the last inequality. Therefore,
				\begin{equation}\label{psi-increment}
					|\Psi(y+x)-\Psi(y)|\leq C_\eta|x|^\eta,\quad x\in[-1,1].
				\end{equation}
				
				Since
				\[\hp^\pr(y)=\p^\pr(y)=-y^{r-1},\quad y\geq 1,\]
				by symmetric property and the fact that $r\leq\al/2<1$, 
				\[|\hp^\pr(y)|\leq |y|^{r-1}\leq C,\quad |y|\geq 1.\] This means $\hp^\pr$ is globally bounded, and by mean value theorem,
				\[ |\hp(y+x)-\hp(y)|\leq C|x|\leq C|x|^\eta,\quad x\in[-1,1], \]
				Combined with \eqref{psi-increment}, (i) is proved.
				
				(ii) By (i),
				\begin{equation}
					\int_{|x|<1}|\T(y+x)-\T(y)|^2\nu(dx)\leq C_\eta\int_{|x|<1}|x|^{2\eta}\nu(dx)\leq C_\eta\int_{|x|<1}|x|^{2\eta-1-\al}dx\leq C_{\eta,\al}
				\end{equation}
				for each $\eta\in (\al/2,1).$ So it remains to show
				\begin{equation}
					\int_{|x|\geq1}|\T(y+x)-\T(y)|^2\nu(dx)\leq M
				\end{equation}
				with $M$ independent of $y$. Since $\Psi$ is bounded and $\nu(|x|\geq 1)<\infty$, we have
				\begin{equation}\label{increment-psi}
					\int_{|x|\geq1}|\Psi(y+x)-\Psi(y)|^2\nu(dx)\leq M.
				\end{equation}
				It remains to show
				\begin{equation}\label{increment-hp}
					\int_{|x|\geq1}|\hp(y+x)-\hp(y)|^2\nu(dx)\leq M.
				\end{equation}
				If $r<0,$ then $\hp$ is also bounded, and \eqref{increment-hp} is proved in precisely the same way as \eqref{increment-psi}. So below we assume $0<r<\al/2.$ Note that $|\hp(z)|\leq C(1+|z|^r)$ for all $z\in\R.$
				
				\noindent\textbf{Case 1: $|y|< 2.$}
				
				Since $|x|\geq 1,$ $|y+x|\leq 2+|x|\leq 2|x|+|x|=3|x|,$ and thus $|\hp(y)|\leq C(1+|y|^r)\leq C(1+|x|^r),$ and
				\[|\hp(y+x)|\leq C(1+|y+x|^r)\leq C|x|^r.\] Hence by the fact $r<\al/2,$
				\begin{equation}
					\int_{|x|\geq 1}|\hp(y+x)-\hp(y)|^2\nu(dx)\leq\int_{|x|\geq 1}C(1+|x|^{2r})\nu(dx)\leq C(1+\int_1^\infty x^{2r-1-\al}dx)\leq C.
				\end{equation} 
				
				\noindent\textbf{Case 2: $|y|\geq 2.$}
				
				Write
				\begin{align}\label{J1+J2}
					&\quad\int_{|x|\geq 1}|\hp(y+x)-\hp(y)|^2\nu(dx)\nonumber\\&=\int_{1\leq|x|\leq |y|/2}|\hp(y+x)-\hp(y)|^2\nu(dx)+\int_{|x|>|y|/2}|\hp(y+x)-\hp(y)|^2\nu(dx)\nonumber\\&=:J_1+J_2.
				\end{align}
				We estimate $J_1$ first. By mean value theorem,
				\[|\hp(y+x)-\hp(y)|=|\hp^\pr(y+\theta x)x|\leq C|y|^{r-1}|x|,\]
				where we used $|\hp^\pr(z)|\leq C|z|^{r-1}$ for $z$ with $|z|\geq 1$ and that
				\[|y+\theta x|\geq |y|-|\theta x|\geq |y|-|x|\geq |y|-\frac{|y|}{2}=\frac{|y|}{2}. \]
				Therefore,
				\begin{equation}\label{J1-estimate}
					\begin{aligned}[b]
						\quad J_1&\leq C\int_{1\leq|x|\leq |y|/2}|y|^{2r-2}|x|^2\nu(dx)\\&= C|y|^{2r-2}\int_{1\leq|x|\leq|y|/2}|x|^{1-\al}dx\\&\leq C|y|^{2r-\al}\leq C.
					\end{aligned}
				\end{equation}
				For $J_2,$ notice that
				\[|\hp(y+x)-\hp(y)|\leq |\hp(y+x)|+|\hp(y)|\leq C(1+|y+x|^r)+C|y|^r\leq C(1+|x|^r+|y|^r).\]
				But $|x|\geq |y|/2\geq 1,$ so $1+|x|^r+|y|^r\leq 1+C|x|^r\leq |x|^r+C|x|^r\leq C|x|^r.$ Therefore,
				\[|\hp(y+x)-\hp(y)|\leq C|x|^r,\]
				from which
				\begin{equation}\label{J2-result}
					J_2\leq C\int_{|x|>|y|/2}|x|^{2r}\nu(dx)=C\int_{|y|/2}^\infty x^{2r-1-\al}dx\leq C|y|^{2r-\al}\leq C.
				\end{equation}
				Combining \eqref{J1+J2}, \eqref{J1-estimate} and \eqref{J2-result}, the case of $|y|\geq 2$  is proved.
			\end{proof}
			
			\begin{lemma}\label{lemma-Y-L}
			Consider the one-dimensional additive SDE
			\begin{equation}\label{Y-zeta}
				dY^\zeta_t=g(Y^\zeta_t)dt+dL_t,\quad Y^\zeta_0=\zeta,
			\end{equation}	
			where $\zeta$ is an almost surely finite random variable.
			For each $u>0,$
			\[\sup_{0\leq s\leq u}|Y^\zeta_s|\leq |\zeta|+2\sup_{0\leq s\leq u}|L_s|.\] 
			\end{lemma}
			\begin{proof}
				Set $Z_s:=Y^\zeta_s-L_s.$ The SDE \eqref{Y-zeta} implies $\dot{Z}_s=g(Z_s+L_s).$ Define \[b:=\sup_{0\leq s\leq u}|L_s|,\quad c_1:=\max\{\zeta,b\},\quad c_2:=\min\{\zeta,-b\}.\]
				By \eqref{Y-zeta}, $\dot{Z}_s=g(Y^\zeta_s)=g(Z_s+L_s)$ for a.e. $s\in [0,u].$ Therefore, $Z$ is absolutely continuous, and  for each  $s\in[0,u],$
				\[(Z_s-c_1)^+=(Z_0-c_1)^++\int_0^s\mathbf{1}_{\{Z_v>c_1\}}\dot{Z}_vdv=\int_0^s\mathbf{1}_{\{Z_v>c_1\}}g(Z_v+L_v)dv,\] where we use $Z_0=\zeta\leq c_1$ in the last equality. Note that whenever $Z_v>c_1,$ \[Z_v+L_v>c_1+L_v\geq c_1-b\geq 0,\] so $g(Z_v+L_v)<0$, which implies $(Z_s-c_1)^+\leq 0.$ That is, $Z_s\leq c_1.$ Similarly, for  $s\in [0,u],$
				\[(c_2-Z_s)^+=(c_2-Z_0)^+-\int_0^s\mathbf{1}_{\{Z_v<c_2\}}\dot{Z}_vdv=-\int_0^s\mathbf{1}_{\{Z_v<c_2\}}g(Z_v+L_v)dv.\]
				Whenever $Z_v<c_2,$ $Z_v+L_v<c_2+L_v\leq c_2+b\leq 0,$ so $g(Z_v+L_v)\geq 0,$ which implies $(c_2-Z_s)^+\leq 0.$ That is, $c_2\leq Z_s.$  Note also
				\[c_1\leq \max\{|\zeta|,b\},\quad c_2\geq -\max\{|\zeta|,b\},\] and this is to say, $|Z_s|\leq \max\{|\zeta|,b\}$ for each $s\in [0,u].$ Recalling that $Y^\zeta_s=Z_s+L_s,$ which yields
				\[|Y^\zeta_s|\leq |Z_s|+|L_s|\leq (|\zeta|+b)+|L_s|\leq |\zeta|+\sup_{0\leq s \leq u}|L_s|+|L_s|.\]
				Taking supremum on $[0,u],$
				\[\sup_{0\leq s\leq u}|Y^\zeta_s|\leq |\zeta|+2\sup_{0\leq s\leq u}|L_s|.\]

			\end{proof}

	
\end{document}